\documentclass[nonatbib]{elsarticle}
\usepackage{amssymb,amsmath,graphicx,amsfonts,euscript,fullpage}
\usepackage{fancyhdr}
\usepackage{epsfig}
\usepackage{setspace}
\usepackage[all]{xy}
\usepackage{amsmath}
\usepackage{graphicx}
\usepackage{caption}
\usepackage{multirow}
\usepackage{verbatim}
\usepackage{fancyhdr}
\usepackage{extramarks}
\usepackage{amsmath}
\usepackage{amssymb}
\usepackage{mathrsfs}
\usepackage{amsthm}
\usepackage{amsfonts}
\usepackage{tikz}
\usepackage[noend]{algpseudocode}
\usepackage{algorithmicx,algorithm}
\usepackage{algpseudocode}
\usepackage{multirow}
\usepackage{subfigure,color}
\usepackage{epstopdf}
\usepackage{graphicx}
\usepackage{animate}
\usepackage{hyperref}
\usepackage{lineno}
\makeatletter
\let\c@author\relax
\makeatother
\usepackage{biblatex}
\graphicspath{{figures/}}

\numberwithin{equation}{section}
\usetikzlibrary{automata,positioning}

\DeclareMathOperator{\conv}{conv}

\newcommand{\E}{\mathrm{E}}

\newtheorem{theorem}{\textbf{Theorem}}
\newtheorem{corollary}{\textbf{Corollary}}

\newtheorem{lemma}{\textbf{Lemma}}

\newtheorem{example}{\textbf{Example}}

 \usepackage{pst-grad}
 \usepackage{pst-plot}
\def \bes{\begin{eqnarray}}
\def \ees{\end{eqnarray}}
\journal{Journal of Computational Physics}

\begin{document}

\begin{frontmatter}

\title{Greedy Training Algorithms for Neural Networks and Applications to PDEs}

\author[1]{Jonathan W. Siegel\corref{cor1}}
\ead{jus1949@psu.edu}
\author[1]{Qingguo Hong}
\author[2]{Xianlin Jin}
\author[1]{Wenrui Hao}
\author[1]{Jinchao Xu}

\cortext[cor1]{Corresponding author.}

\address[1]{Department of Mathematics, Pennsylvania State University, University Park, PA, 16802, USA}

\address[2]{School of Mathematical Sciences, Peking University, Beijing, China}

\begin{abstract}
Recently, neural networks have been widely applied for solving partial differential equations (PDEs). Although such methods have been proven remarkably successful on practical engineering problems, they have not been shown, theoretically or empirically, to converge to the underlying PDE solution with arbitrarily high accuracy. The primary difficulty lies in solving the highly non-convex optimization problems resulting from the neural network discretization, which are difficult to treat both theoretically and practically. It is our goal in this work to take a step toward remedying this. For this purpose, we develop a novel greedy training algorithm for shallow neural networks. Our method is applicable to both the variational formulation of the PDE and also to the residual minimization formulation pioneered by physics informed neural networks (PINNs). We analyze the method and obtain a priori error bounds when solving PDEs from the function class defined by shallow networks, which rigorously establishes the convergence of the method as the network size increases. Finally, we test the algorithm on several benchmark examples, including high dimensional PDEs, to confirm the theoretical convergence rate. Although the method is expensive relative to traditional approaches such as finite element methods, we view this work as a proof of concept for neural network-based methods, which shows that numerical methods based upon neural networks can be shown to rigorously converge.%
\end{abstract}


\end{frontmatter}

\section{Introduction}
Machine learning based approaches in the computational mathematics community have increased rapidly in recent years. One of the main new applications of machine learning  has been to the numerical solution of differential equations. 
In particular, neural network-based discretization has become a revolutionary tool for solving differential equations \cite{dissanayake1994neural,sirignano2018dgm,han2018solving} and for learning the underlying physics behind experimental data \cite{raissi2019physics}. This approach has been applied to a wide variety of practical problems with astounding success \cite{cai2022physics,cai2021physics,mao2020physics,sahli2020physics,pang2019fpinns}.  The benefit of this new approach is that neural networks can lessen or even overcome the curse of dimensionality for high-dimensional problems \cite{khoo2021solving,han2018solving,grohs2018proof,lanthaler2021error}. This is due to the dimension independent approximation properties of neural networks \cite{barron1993universal,klusowski2018approximation}, which have been compared with finite element methods (FEMs) and other tradition methods in approximation theory
\cite{CiCP-28-1707,siegel2020high,devore2020neural,daubechies2021nonlinear,yarotsky2017error,lu2021deep,shen2022optimal}. 

Broadly speaking, there are three approaches for solving PDEs using neural networks which have been extensively studied recently. The first approach is to use neural networks to parameterize a set of functions in which the PDE is solved. This approach is taken by the deep Ritz method \cite{weinan2018deep} and by physics informed neural networks (PINNs) \cite{raissi2019physics}, and has been used to effectively solve the Schr\"odinger equation \cite{hermann2020deep,carleo2017solving}. 
Another common approach is to learn the PDE solution operator using neural networks. This allows the efficient approximation of new solutions as parameters of the underlying PDE are varied and is a non-linear analogue of model reduction \cite{sawant2021physics}. The effectiveness of this approach has been shown through the success of DeepONet \cite{lu2021learning}, the Fourier neural operator \cite{li2020fourier}, and the Galerkin Transformer \cite{cao2021choose}, which have recently been theoretically analyzed \cite{kovachki2021universal,lanthaler2021error}. Finally, there is the approach of learning the underlying PDE itself from data using deep neural networks, which was pioneered by PINNs \cite{raissi2019physics}. In the following, we consider exclusively the first approach, where neural networks are used to parameterize a set of functions in which the equation is solved.

Generally speaking, we classify the numerical error of the neural network discretization into three parts: 1) the modeling error incurred by solving the PDE over a restricted function class; 2) the optimization error incurred by failing to fully optimize over the function class; and 3) discretization error incurred by discretizing the integrals appearing in the weak form of the equation (More details in Section 3). There are some results which bound the modeling error by considering how efficiently the PDE solution can be approximated with neural networks. For instance, the convergence analysis of the finite neuron method is discussed by considering a family of $H^m$-conforming piecewise polynomials based on artificial neural network \cite{CiCP-28-1707}. The convergence rate of the deep Ritz method depends on the dimensionality \cite{duan2021convergence}. 
The error estimate of the deep Ritz Method for elliptic problems with different boundary conditions is established in \cite{muller2021error}. 
The convergence analysis of  the least-squares method based on residual minimization in PINNs has been studied in \cite{shin2020error} based on strong and variational formulations. The convergence of PINNs to the  PDE solution is 
analyzed in \cite{shin2020convergence}
for linear second-order elliptic and parabolic  PDEs.

The optimization error arises when the highly non-linear and non-convex optimization problem resulting from discretizing using neural networks is only approximately solved. There have been some results in the literature which work toward bounding this error. For instance, it has been shown that gradient descent applied to a sufficiently wide network will reach a global minimum \cite{luo2020two,du2018gradient,allen2019convergence,zou2020gradient,arora2019fine}. In addition, the convergence of stochastic gradient descent (SGD) and Adam \cite{kingma2014adam} has been analyzed in Fourier space. This results in the empirical observation that the error converges fastest in the lowest frequency modes which is known as the frequency principle or spectral bias of neural network training \cite{luo2019theory,rahaman2019spectral}. In practice, Adam or SGD are typically used to solve the resulting optimization problems, although other methods, such as a randomized Newton's method \cite{chen2019randomized} and novel specialized methods, for instance the Active Neuron Least Squares method \cite{ainsworth2022active}, have also been explored. Recently, an interesting optimization method which resembles the greedy algorithms we introduce has been developed for shallow ReLU neural networks \cite{ainsworth2022active}. %

However, the important point is that none of these algorithms empirically achieve asymptotic convergence as the network size increases \cite{weinan2018deep,raissi2019physics}. More specifically, the relative $L_2$ error of the deep Ritz method using the SGD optimizer stabilizes as the number of neurons increases (Table 1 in \cite{weinan2018deep}), and the relative $L_2$ error of PINNs using the L-BFGS optimizer even increases as the number of neurons increases (Tables A.2 \& A.3 in \cite{raissi2019physics}).
Moreover, for one-hidden-layer neural network with fixed inner weights with the ReLU activation function, one can prove that the condition number of the mass matrix is $\mathcal{O}(n^4)$, where $n$ is the number of neurons \cite{hong2022activation}. This implies that gradient descent method converges very slow especially when $n$ is large. We want to stress that this lack of convergence in no way diminishes the practical utility of PINNs and the deep Ritz methods. In many practical problems, these methods achieve more than sufficient accuracy. However, from a mathematical point of view the question of whether neural network methods can be used to provably solve differential equations remains interesting.

Concerning the generalization accuracy,
there are also some analytical results along this direction. For instance, a priori generalization analysis of the deep Ritz method is studied using the Barron norm with activation function ${\rm SP}_\tau$ in \cite{lu2021priori}. The empirical risk of the PDE solution represented by an over-parameterized-two-layer neural network achieves a global minimizer under some assumptions
\cite{luo2020two}. The generalization error of PINNs can be bounded by the training error \cite{mishra2020estimates}. 
The generalization error of deep learning--based methods is also analyzed for high dimensional Black-Scholes PDEs to overcome the curse of dimensionality in
\cite{berner2020analysis}.

However, there are significant gaps in the existing convergence and generalization theory. In particular, the wide networks which are required to make gradient descent or SGD converge cannot be guaranteed to generalize well. On the other hand, networks which are small enough or satisfy an appropriate bound on their coefficients to guarantee generalization cannot be provably optimized using gradient descent or its variants. Recently, this gap has been closed for shallow neural networks in \cite{hong2021rademacher}. 

To control these three numerical errors and observe the  asymptotic convergence order numerically,
we propose provably convergent algorithms in this paper for efficiently solving the neural network optimization problem. The key idea is to use a greedy algorithm to train shallow neural networks instead of gradient descent. Greedy algorithms have previously been proposed for solving PDEs using a basis of separable functions \cite{figueroa2012greedy,cances2013greedy,ammar2006new,le2009results}, and have been proposed for training shallow neural networks \cite{lee1996efficient}. Our contributions are to develop a convergence analysis when using greedy algorithms for training shallow neural networks to solve PDEs, to show the practical feasibility of this method even in high dimensions, and to demonstrate that the theoretically derived convergence rates are achieved. To the best of our knowledge, this work is the first rigorous analysis without gaps which uses neural networks to solve PDEs and also the first neural network training algorithm which observes asymptotic convergence numerically. Although the method is currently not particularly efficient, we view it as a proof-of-concept which demonstrates the viability of using neural networks to rigorously solve PDEs. %
Improving the efficiency of the method and extending it to deeper networks with more complex architectures is a promising future research direction.

The remaining part of the paper is organized as follows: in Section \ref{basic-section}, we introduce the problem setup and discuss the class of elliptic PDEs we will solve. We overview the basic machine learning theory for PDEs in Section \ref{basis_machine_learning_theory} and introduce the neural network model classes in Section \ref{sec:model-class}, where we also discuss the approximation error associated with using a neural network discretization. In Section \ref{sec:greedy-algorithms}, we discuss greedy algorithms for non-linear dictionary approximation and their convergence analysis, which bounds the optimization error. In section \ref{sec:uniform_error}, we show how to bound the discretization error when discretizing the PDE energy. 
Several numerical examples are used to demonstrate the efficiency of the greedy algorithms in Section \ref{sec:numerical} and finally a conclusion is given in Section \ref{sec:con}.

\section{Basic setup and the model problem}\label{basic-section}
\subsection{Variational Formulation}
We follow here largely the setting in \cite{CiCP-28-1707}. Let $\Omega\subset \mathbb{R}^d$ be a bounded domain with a
sufficiently smooth boundary $\partial\Omega$.  For any integer $m\ge
1$, we consider the following $2m$-th order partial differential
equation with certain boundary conditions:
\begin{equation} \label{2mPDE}
\left\{
  \begin{array}{rccl}\displaystyle
Lu &=& f &\mbox{in }\Omega, \\
B^k(u) &= &0 & \mbox{on }\partial\Omega \quad(0\le k\le m-1),
  \end{array}
\right.
\end{equation}
where $B^k(u)$ denotes the Dirichlet, Neumann, or mixed boundary conditions which will be discussed in detail in the following. Here $L$ is the partial differential operator defined as follows
\begin{equation}\label{Lu}
Lu= \sum_{|\alpha|=m}(-1)^m\partial^\alpha (a_\alpha(x)\,\partial^\alpha\,u) +a_0(x)u,
 \end{equation} 
where $\alpha$ denotes $n$-dimensional multi-index $\alpha
= (\alpha_1, \cdots, \alpha_n)$ with
$$
|\alpha|= \sum_{i=1}^n \alpha_i, \quad
\partial^\alpha = \frac{\partial^{|\alpha|}}{\partial x_1^{\alpha_1}
\cdots \partial x_n^{\alpha_n}}.
$$
For simplicity, we assume that $a_\alpha$ are strictly positive and bounded on
$\Omega$ for $|\alpha|=m$ and $\alpha=0$, namely, $\exists \alpha_0>0, \alpha_1 < \infty$, such that 
\begin{equation}\label{ass:1}
\alpha_0\leq a_\alpha(x), a_0(x) \leq \alpha_1\,\, \forall x\in\Omega,\,\,|\alpha|=m.
 \end{equation}
 Further, when considering deterministic numerical quadrature in Section \ref{sec:num_qua}, we will make the additional assumption that $a_\alpha$ are sufficiently smooth. %
 
Given a nonnegative integer
$k$ and a bounded domain $\Omega\subset \mathbb{R}^d$, let 
\begin{equation}\label{Hkspace}
H^k(\Omega):=\left\{v\in L^2(\Omega), \partial^\alpha v\in L^2(\Omega), |\alpha|\le k\right\}
\end{equation}
be standard Sobolev spaces with norm and seminorm given respectively by 
$$
 \|v\|_k:=\left(\sum_{|\alpha|\le k} \|\partial^\alpha v\|_0^2\right)^{1/2}, \quad  |v|_k:=\left(\sum_{|\alpha|= k} \|\partial^\alpha v\|_0^2\right)^{1/2}.
$$
For $k=0$, $H^0(\Omega)$ is the standard $L^2(\Omega)$ space
with the inner product denoted by $(\cdot, \cdot)_{0,\Omega}$.  Similarly, for
any subset $K\subset \Omega$, $L^2(K)$ inner product is denoted by
$(\cdot, \cdot)_{0,K}$.  We note that, by a well-known property of Sobolev spaces, the assumption \eqref{ass:1} implies that
\begin{equation}
  \label{avv}
a(v,v)\gtrsim \|v\|^2_{m,\Omega}, \forall v\in H^m(\Omega),
\end{equation}
where $\displaystyle a(u,v) := \sum_{|\alpha | = m}(a_\alpha\partial^{\alpha}u, \partial^{\alpha}v)_{0,\Omega} +(a_0u,v)$.

Next, we discuss the boundary conditions in detail. A popular type of boundary conditions is the Dirichlet boundary
condition when $B^k=B_D^k$ are given by the following Dirichlet type
trace operators
\begin{equation}\label{BD}
B_D^k(u):=\left.\frac{\partial^k u}{\partial
    \nu^k}\right|_{\partial\Omega}\quad (0\le k\le m-1),
\end{equation}
with $\nu$ being the outward unit normal vector of $\partial\Omega$. Using (\ref{BD}), we define 
\begin{equation}
H_0^m(\Omega) = \lbrace v \in H^m(\Omega): B_D^k(v) = 0, 0\leq k\leq m-1 \rbrace.
\end{equation}

For the aforementioned Dirichlet boundary condition, the elliptic
boundary value problem \eqref{2mPDE} is equivalent to 

\begin{description}
\item[Minimization Problem D:] Find $u\in H^m_0(\Omega)$ such that 
\begin{equation}
\label{minJv}
u=\arg\min_{v\in H^m_0(\Omega)} \mathcal{R}(v),
\end{equation}
\end{description}
with the energy function $\mathcal{R}$ defined by
\begin{equation}\label{energy-function}
\mathcal{R}(v)=\frac12 a(v,v) -\int_{\Omega} fv dx.
\end{equation}
Next we consider the following minimization problem 
over the whole space $H^m(\Omega)$
\begin{description}
\item[Minimization Problem N:] Find $u\in H^m(\Omega)$ such that 
\begin{equation}\label{minproblem}
 u = \arg\min_{v\in H^m(\Omega)} \mathcal{R}(v),
\end{equation}
\end{description}
with energy function $\mathcal{R}$ defined by \eqref{energy-function}.

\begin{comment}
In this sense, the pure Neumann boundary conditions are the most natural and easiest to enforce, especially when optimizing over the class of neural network functions. It remains to be determined which form the pure Neumann boundary conditions, which we denote by $B_N^k$, take when $m\geq 2$. This is given by the following result, which applies to any PDE operator \eqref{Lu}.

for general when $m\ge 2$.  We first begin our discussion
with the following simple result. 

\begin{lemma}\cite[Lemma 5.1]{CiCP-28-1707} \label{lem:BDBNdual}
For each $k=0,1,\ldots,m-1$, there exists a bounded linear differential operator of order $2m-k-1$:
\begin{equation}
    \label{BN}
B_N^k: H^{2m}(\Omega)\mapsto L^2(\partial\Omega)    
  \end{equation}
such that the following identity holds
$$
(Lu,v)=a(u,v)-\sum_{k=0}^{m-1}\langle B_N^k(u),B_D^k(v)\rangle _{0,\partial\Omega}.
$$
Namely
\begin{equation}\label{BDBNdual}
\sum_{|\alpha|=m}(-1)^m\left(\partial^\alpha
  (a_\alpha\,\partial^\alpha\,u),\,v\right) _{0,\Omega}
=\sum_{|\alpha|=m}\left(a_\alpha\partial^\alpha\,u, \partial^\alpha v\right) _{0,\Omega}
-\sum_{k=0}^{m-1}\langle B_N^k(u),B_D^k(v)\rangle _{0,\partial\Omega}
\end{equation}
for all $u\in H^{2m}(\Omega), v\in  H^{m}(\Omega)$. Furthermore, 
\begin{equation}\label{equ:regassum}
\sum_{k=0}^{m-1}\|B_D^k(u)\|_{L^2(\partial\Omega)}+\sum_{k=0}^{m-1}\|B_N^k(u)\|_{L^2(\partial\Omega)}\lesssim \|u\|_{2m, \Omega}.
\end{equation}
\end{lemma}
\end{comment}

%
%
The optimization problem \eqref{minproblem} is equivalent to the following pure Neumann boundary value problems for the PDE operator \eqref{Lu}:
\begin{equation}\label{model-problem}
\left\{
  \begin{array}{rccl}
Lu &=& f &\mbox{in }\Omega, \\
B_{N}^k(u) &= &0 & \mbox{on }\partial\Omega \quad(0\le k\le m-1),
  \end{array}
\right.
\end{equation}
where
\begin{equation}
    \label{BN}
B_N^k: H^{2m}(\Omega)\mapsto L^2(\partial\Omega)    
  \end{equation}
such that the following identity holds
\begin{equation}
(Lu,v)=a(u,v)-\sum_{k=0}^{m-1}\langle B_N^k(u),B_D^k(v)\rangle _{0,\partial\Omega}.
  \end{equation}
In particular,
\begin{itemize}
\item For $m=1$, we have $B_N^k(u) = \frac{\partial u}{\partial n}$.
\item For $m=2$ and $d=2$, we have $B_{N}^{0} u=\left.\frac{\partial}{\partial n}\left(\Delta u+\frac{\partial^{2} u}{\partial s^{2}}\right)-\frac{\partial}{\partial s}\left(\kappa_{s} \frac{\partial u}{\partial s}\right) \right|_{\partial \Omega} \text { and } B_{N}^{1} u=\left.\frac{\partial^{2} u}{\partial n^{2}}\right|_{\partial \Omega}$.
\end{itemize}
In order to handle Dirichlet boundary conditions, we consider the mixed boundary value problem:
\begin{equation} \label{equ:delta}
\left\{
\begin{aligned}
Lu_{\delta} &= f \qquad \mbox{in }\Omega, \\
B_D^k(u_{\delta})+\delta B_N^k(u_\delta)  &= 0, \ \ 0\le k\le m-1.
\end{aligned}
\right.
\end{equation}
It is easy to see that \eqref{equ:delta} is equivalent to the following optimization problem:
\begin{equation}\label{equ:varpdelta}
u_{\delta}=\arg\min_{v\in H^m(\Omega)} \mathcal{R}_{\delta}(v)
\end{equation}
where
\begin{equation}\label{energy:function:D}
 \mathcal{R}_{\delta}(v)= {1\over 2}a_\delta(v,v)-(f,v)
 \end{equation}
 and
\begin{equation}\label{a-delta}
a_\delta(u,v)=a(u,v)+\delta^{-1}\sum_{k=0}^{m-1}\langle B_D^k(u), B_D^k(v)\rangle_{0,\partial\Omega}.
\end{equation}
 Using the theory developed in \cite{CiCP-28-1707}, we have an estimate between $u$ and $u_\delta$ as follows

\begin{lemma}\cite[Lemma 5.4]{CiCP-28-1707}\label{Nitchtrick}
Define $\|\cdot\|_{a,\delta} = \sqrt{a_{\delta}(\cdot, \cdot)}$. Let $u$ be the solution of (\ref{2mPDE}) with $B^k = B_D^k$, $0\leq k\leq m-1$, and $u_{\delta}$ be the solution of \eqref{equ:delta}. Then 
\begin{equation}
    \|u-u_{\delta}\|_{a,\delta} \lesssim \sqrt{\delta} \|u\|_{2m, \Omega}.
\end{equation}
\end{lemma}
\subsection{Residual Formulation}
The second type of problem we will consider are more general (potentially) non-elliptic and non-symmetric linear PDEs given by
\begin{equation}
    Lu = f~\text{in}~\Omega,
\end{equation}
where the operator $L$ is given by
\begin{equation}
    Lu = \sum_{|\alpha| \leq m} a_\alpha(x)\partial^\alpha u.
\end{equation}
We approach such an equation using the reidual minimization technique poineered by physics informed neural networks (PINNs) \cite{raissi2019physics}. This approach has us minimizing the residual norm
\begin{equation}\label{PINN-risk}
    \mathcal{R}(v) = \frac{1}{2}\int_{\Omega}(Lv - f)^2dx + \int_{\partial\Omega} B(v)dx,
\end{equation}
where $B(v)$ are our boundary conditions.

The advantage of the PINNs approach is exceptional flexibility which allows arbitrary equations, boundary conditions, data assimilation, and unknown terms in the equation itself to be treated in a straightforward manner which can be implemented rapidly. This flexibility has driven multiple recent breakthroughs in scientific computing \cite{cai2022physics,cai2021physics,mao2020physics,sahli2020physics,pang2019fpinns}.

Our theory will allow us to obtain both a priori and a posteriori bounds on the PDE residual $\mathcal{R}(v)$. Relating this to the solution error is an important problem which has been studied for a variety of PDEs under certain assumptions, including for linear elliptic and parabolic PDEs \cite{shin2020convergence,mishra2022estimates}, for Kolmogorov PDEs \cite{de2021error}, and for the Navier-Stokes equation \cite{de2022error}.

\section{Basic machine learning theory for PDEs}\label{basis_machine_learning_theory}
In this section, we describe the basics of machine learning and statistical learning theory and explain their connections with numerical methods for solving PDEs. Our focus will be on the connections with numerical PDEs, while the statistics and probability theory background can be found in standard references on statistical learning theory \cite{shalev2014understanding,mohri2018foundations}.
\subsection{General objective}
We consider the following general setup corresponding to classification or regression. Let $X, Y$ and $Z$ denote three sets. Here $X$ represents the input space, $Y$ the label space, and $Z$ is the prediction space. We are trying to `learn' a function $u:X\rightarrow Z$. We suppose that $u$ minimizes the risk, defined by
\begin{equation}\label{risk-definition}
    u = \arg\min_{v\in \mathcal{F}}\mathcal{R}(v),~\text{where } \mathcal{R}(v) = \mathbb{E}_{x,y\sim d\mu}[l(x, y, v(x))] = \int_{X\times Y} l(x, y, v(x))d\mu(x,y),
\end{equation}
\begin{comment}
\begin{equation}\label{risk-definition}
    u = \arg\min_{v\in \mathcal{F}}\mathcal{R}(v) = \arg\min_{v\in \mathcal{F}}\mathbb{E}_{x,y\sim d\mu}[l(x, y, v(x))] = \arg\min_{v\in \mathcal{F}}\int_{X\times Y} l(x, y, v(x))d\mu(x,y),
\end{equation}
\end{comment}
over an appropriate function class $\mathcal{F}$. Here $l:X\times Y\times Z\rightarrow \mathbb{R}$ is an appropriate loss function, and $d\mu$ is a probability measure on $X\times Y$. 

For example, in a binary image classification problem we would set $X = [0,1]^{n\times n}$ and $Y = Z = \{0,1\}.$ Here $X$ represents the set of possible $n\times n$ pixel arrangements, i.e. images, and $Y$ and $Z$ represent the two possible classes. The function $u:X\rightarrow Y$ maps an image $x$ to a label $u(x)\in \{0,1\}$. A typical loss function would be the indicator function 
\begin{equation}\label{classification-loss}
    l(x,y,z) = y(1-z)+z(1-y) = \begin{cases}
            0 & y = z,\\
            1 & y\neq z.
            \end{cases}
\end{equation}
In this case the risk \eqref{risk-definition} is exactly the classification error, since we calculate
\begin{equation}
    \mathcal{R}(v) = \mathbb{E}_{(x,y)\sim d\mu}[l(x, y, v(x))] = \mathbb{P}_{(x,y)\sim d\mu}[y\neq v(x)].
\end{equation}
The function class $\mathcal{F}$ could be taken as the set of all measurable functions from $X$ to $Z$, for instance.

To give another example which is more closely related to the situation when solving PDEs, we consider a regression problem, where $X = \mathbb{R}^d$ is the space of regressors, and $Y = Z = \mathbb{R}$ is the space of responses. In this case, we would take
\begin{equation}
    l(x,y,z) = \frac{1}{2}(y-z)^2,
\end{equation}
for instance. In this case the risk is exactly the expected $\ell^2$ regression error
\begin{equation}
    \mathcal{R}(v) = \mathbb{E}_{(x,y)\sim d\mu}[l(x, y, v(x))] = \frac{1}{2}\int_{\mathbb{R}^d\times \mathbb{R}} |y - v(x)|^2d\mu(x,y).
\end{equation}

To put the solution of PDEs into this framework, we let $X = \Omega \subset\mathbb{R}^d$, $Y = \{0\}$ (i.e. we have no labels) and $Z = \mathbb{R}$, and consider the function class $\mathcal{F} = H^m(\Omega)$. The distribution $d\mu$ on $X\times Y$, which we can simply identify with $X = \Omega$, is the uniform distribution on the domain $\Omega$. We frame the solution of the PDE as the minimization of the risk \eqref{risk-definition} for an appropriate loss function $l$. For the solution of PDEs, the loss function must depend upon the derivatives of $u$, so we consider the somewhat more general risk

\begin{equation}\label{PDE-risk}
    \mathcal{R}(v) = \mathbb{E}_{x\sim d\mu}[l(x,v(x),D v(x), ..., D^mv(x))] = \int_{X} l(x,v(x),D v(x), ..., D^mv(x))d\mu(x).
\end{equation}
There are two prominent approaches for framing a PDE in this manner. One, known as the deep Ritz method \cite{weinan2018deep} is to consider the variational formulation of the PDE. For the elliptic PDE \eqref{2mPDE}, this corresponds to setting
\begin{equation}\label{elliptic-loss}
    l(x,v(x),D v(x), ..., D^kv(x)) = \frac{1}{2}\left(\sum_{|\alpha| = m} a_\alpha(x)|\partial^\alpha v(x)|^2 + a_0(x)v(x)^2\right) - f(x)v(x)
\end{equation}
 to solve the $2m$-th order elliptic equation \eqref{2mPDE}. In the case of Dirichlet boundary conditions, we must add to this an expectation of an appropriate penalty over the boundary of the domain $X = \Omega$, i.e. our risk becomes
\begin{equation}\label{dirichlet-risk}
    \mathcal{R}(v) = \int_{X} l(x,v(x),D v(x), ..., D^mv(x))d\mu(x) + \delta^{-1}\int_{\partial X} l_{BC}(x,v(x),D v(x), ..., D^{m-1}v(x))d\mu_{BC}(x).
\end{equation}
Here the loss function for the boundary conditions is given by
\begin{equation}
    l_{BC}(x,v(x),D v(x), ..., D^{m-1}v(x)) = \sum_{k=0}^{m-1} B_D^k(v(x))^2 = \sum_{k=0}^{m-1} \left(\frac{\partial^k v}{\partial\nu^k}(x)\right)^2,
\end{equation}
where $\nu$ denotes the outward normal vector. The distribution $\mu_{BC}$ is the uniform distribution on the boundary of the domain $\Omega$.

The other main approach we consider, which was pioneered in the breakthrough work on physics informed neural networks (PINNs) \cite{raissi2019physics}, sets the loss function to the $L^2$ residual of the PDE, i.e. in order to solve the $m$-th order equation $Lu = f$, we set our loss function to
\begin{equation}\label{PINN-loss}
    l(x,v(x),D v(x), ..., D^mv(x)) = \frac{1}{2}(Lu - f)^2 = \frac{1}{2}\left(\sum_{|\alpha| \leq m} a_\alpha(x)\partial^\alpha v(x) - f(x)\right)^2,
\end{equation}
which results in the risk \eqref{PINN-risk} when appropriate boundary conditions are added.

\subsection{A priori bounds and statistical learning theory}\label{a-priori-section}
Our goal in the work is to design a method for solving PDEs using shallow neural networks which permits \textit{a priori} estimates. Such a method has the property that it can be guaranteed to work as long as the true solution is well-approximated by a given function class. The field of statistical learning theory is concerned with deriving such a priori error estimates for different machine learning methods.

The basic framework of statistical learning theory analyzes the empirical risk minimization procedure. In this method, we draw samples and minimize a potentially modified empirical risk over a restricted function class $\mathcal{F}_\Theta$ (depending upon a set of parameters $\Theta$) to obtain the estimate
\begin{equation}\label{eq-593}
    u_{\Theta,N} = \arg\min_{v\in \mathcal{F}_\Theta} \mathcal{R}_N(v), ~\text{where } \mathcal{R}_N(v)= \frac{1}{N}\sum_{i=1}^N l'(x_i,v(x_i),z_i).
\end{equation}
Here the modified loss function $l'$ is not necessarily the same as loss function $l$ occurring in \eqref{PDE-risk}. This is because the loss function $l$ may not be differentiable or even continuous, which makes the numerical optimization of the empirical risk \eqref{eq-593} intractable. 
For example, the classification loss in \eqref{classification-loss} is discontinuous and this presents significant problems when optimizing. As a result, the classification loss may be replaced by a soft margin SVM loss
\begin{equation}
    l'(x,y,z) = \max(0,1-yz),
\end{equation}
where the model output $z\in Z = \mathbb{R}$ and the label $y\in \{\pm 1\}$. In this case the prediction is not a label, but rather a real number, which can be converted into a label via thresholding. It is easily verified that the SVM loss is a convex upper bound on the classification loss \eqref{classification-loss}.

When solving PDEs, the loss function is continuously differentiable so we usually set $l' = l$. In addition, the distribution $d\mu$ is known explicitly and the empirical risk can be approximated using numerical quadrature instead of sampling (recall that we have no labels $z$ in this case as well)
 \begin{equation}
    u_{\Theta,N} = \arg\min_{v\in \mathcal{F}_\Theta} \mathcal{R}_N(v), ~\text{where } \mathcal{R}_N(v)= \sum_{i=1}^N w_i l(x_i,v(x_i)).
\end{equation}
where $\omega_i$ and $x_i$ are quadrature weights and points in domain $\Omega$. 
When solving equations with Dirichlet boundary conditions, we also need to discretize the integral on the boundary occurring the definition of the risk \eqref{dirichlet-risk}. In this case, our empirical risk would become
\begin{equation}\label{eq-507}
    \mathcal{R}_{N,\delta}(v) = \sum_{i=1}^Nw_il(x_i,v(x_i),z_i) + \delta^{-1}\sum_{i=1}^{N_0}\tilde{w}_il_{BC}(\tilde{x}_i,v(\tilde{x}_i)),
\end{equation}
where the $\tilde{w}_i$ and $\tilde{x}_i$ are quadrature weights and points on the boundary of the domain $\Omega$. Our notation here contains the case where a Monte Carlo discretization is used. In this case the weights $w_i = 1/N$ and the $x_i$ are randomly sampled from a distribution $d\mu$.

In practice, some algorithm is used to approximately solve the optimization problem \eqref{eq-593} to obtain an estimate $\bar{u}_{\Theta,N}$. The risk can then be bounded as
\begin{equation}\label{eq-426}
\begin{split}
    \mathcal{R}(\bar{u}_{\Theta,N}) - \mathcal{R}(u) =~ &[\mathcal{R}(\bar{u}_{\Theta,N}) - \mathcal{R}_N(\bar{u}_{\Theta,N})] + [\mathcal{R}_N(\bar{u}_{\Theta,N}) - \mathcal{R}_N(u_{\Theta,N})]~+ \\
    &[\mathcal{R}_N(u_{\Theta,N}) - \mathcal{R}_N(u_\Theta)] + [\mathcal{R}_N(u_\Theta) - \mathcal{R}(u_\Theta)] + [\mathcal{R}(u_\Theta) - \mathcal{R}(u)],
\end{split}
\end{equation}
where $\displaystyle u_\Theta = \arg\min_{v\in \mathcal{F}_\Theta} \mathcal{R}(v)$ is the minimizer of the true risk over the function class $\mathcal{F}_\Theta$ and $u$ is the global minimizer of the risk (i.e. the function we are trying to learn).

We bound the first and fourth terms in \eqref{eq-426} by
\begin{equation}
    |\mathcal{R}(\bar{u}_{\Theta,N}) - \mathcal{R}_N(\bar{u}_{\Theta,N})| + |\mathcal{R}_N(u_{\Theta}) - \mathcal{R}(u_\Theta)| \leq 2\sup_{v\in \mathcal{F}_\Theta} |\mathcal{R}(v) - \mathcal{R}_N(v)|
\end{equation}
and note that the term $\mathcal{R}_N(u_{\Theta,N}) - \mathcal{R}_N(u_\Theta)$ is non-positive by definition to obtain the following fundamental theorem.
\begin{theorem}\label{trivial-risk-bound-theorem}
    The true risk (also called generalization error) is bounded by
    \begin{equation}\label{generalization-error-decompose}
    \mathcal{R}(\bar{u}_{\Theta,N}) - \mathcal{R}(u) \leq \mathcal{R}(u_\Theta) - \mathcal{R}(u) + 2\sup_{u\in \mathcal{F}_\Theta} |\mathcal{R}(u) - \mathcal{R}_N(u)| + \mathcal{R}_N(\bar{u}_{\Theta,N}) - \mathcal{R}_N(u_{\Theta,N}).
\end{equation}
When using Monte Carlo sampling to discretize the risk, we take an expectation over the samples $x_1,...,x_N$ on both sides of the above equation to get
\begin{equation}
    \begin{split}
        &\mathbb{E}_{x_1,...,x_N} [\mathcal{R}(\bar{u}_{\Theta,N}) - \mathcal{R}(u)] \\
        &\leq\underbrace{\mathcal{R}(u_\Theta) - \mathcal{R}(u)}_{modelling~error} + \mathbb{E}_{x_1,...,x_N} \big[\underbrace{2\sup_{u\in \mathcal{F}_\Theta} |\mathcal{R}(u) - \mathcal{R}_N(u)|}_{discretization~error}\big]
        + \mathbb{E}_{x_1,...,x_N}[\underbrace{\mathcal{R}_N(\bar{u}_{\Theta,N}) - \mathcal{R}_N(u_{\Theta,N})}_{optimization~error}]
    \end{split}
\end{equation}
\end{theorem}
The term on the left hand side here is the \textit{generalization error} which we are trying to bound. We will proceed to analyze the three terms on the right hand side. 

The term $\mathcal{R}_N(\bar{u}_{\Theta,N}) - \mathcal{R}_N(u_{\Theta,N})$ is the \textit{optimization error} of the method. This measures the failure to completely optimize over the model class $\mathcal{F}_{\Theta}$. In traditional methods for solving PDEs, for example finite element methods, this term corresponds to the error in solving the discrete linear system.

The middle term $2\sup_{u\in \mathcal{F}_\Theta} |\mathcal{R}(u) - \mathcal{R}_N(u)|$ is called the \textit{discretization error} and measures the error incurred by discretizing the integral defining the risk \eqref{risk-definition}. In the theory of linear finite elements this term corresponds to numerical quadrature error, which is typcially bounded using Strang's lemma \cite{strang1972variational}. When using a non-linear model class $\mathcal{F}_\Theta$, we must develop new methods for bounding this term. The key tool in our analysis is the Rademacher complexity \cite{bartlett2002rademacher}.

Finally, the term $\mathcal{R}(u_\Theta) - \mathcal{R}(u)$ is called the \textit{modelling error} and measures the failure of the model class $\mathcal{F}_\Theta$ to capture the true solution, or ground truth $u$. In statistical learning theory, this term cannot be theoretically controlled since the ground truth is unknown. The validity of this assumption is checked experimentally either by calculating the empirical risk $\mathcal{R}_N(\bar{u}_{\Theta,N})$ of the learned model or by using a new test dataset if the discretization error cannot be bounded. The advantage of being able to bound the other error terms is that one can conclude that if the method does not empirically perform well on the given data, then this must be due to the model class $\mathcal{F}_\Theta$ not accurately capturing the ground truth.

Bounding this term in the PDE context requires an estimate on how accurately the model class $\mathcal{F}_\Theta$ can approximate the solution of the PDE. This requires both a regularity result on the solution of the PDE and an approximation theoretic result concerning the model class $\mathcal{F}_\Theta$. For the Barron space model class we introduce in Section \ref{sec:model-class} such bounds have been obtained in \cite{CiCP-28-1707,lu2021priori,chen2022regularity,chen2021representation} for certain equations. In addition, sharp approximation results for neural networks on the Barron space can be found in \cite{siegel2021sharp,siegel2020approximation}.

In typical applications of deep learning, including to PDEs \cite{raissi2019physics,weinan2018deep} the empirical risk \eqref{eq-593} is minimized using stochastic gradient descent (SGD) or a variant like ADAM \cite{kingma2014adam}. Bounding both the optimization and discretization error for such methods is a significant challenge. There are results which bound the optimization error by showing that sufficiently large neural networks can be trained to match arbitrary training data using SGD \cite{arora2019fine,du2019gradient}. However, when using such a large network the function class $\mathcal{F}_\Theta$ is very large and this precludes the estimation of the discretization error. This makes analyzing the solution error when solving PDEs using neural networks a significant challenge if SGD or ADAM are used for training. Indeed, a convergence of the error as the size of the network increases cannot be found empirically when solving PDEs \cite{weinan2018deep}, although these methods have reliably been able to attain an acceptable accuracy for many practical problems \cite{cai2022physics,cai2021physics,mao2020physics,sahli2020physics,pang2019fpinns}. Our approach to this problem is to use greedy algorithms for training instead of SGD or ADAM. This allows us to obtain \textit{a priori} estimates on our error, i.e. to bound the optimization and discretization errors. 

\subsection{Test error bounds}\label{a-posteriori-bounds-sec}
Next, we consider the problem of obtaining bounds on the risk $\mathcal{R}$ when Theorem \ref{trivial-risk-bound-theorem} does not apply. 
Suppose that a function $u^*$ has been obtained in some manner, potentially via an unknown black-box method. Our goal is to estimate the risk $\mathcal{R}(u^*)$, i.e. to test the single function $u^*$. Such a situation would occur when we are unable to bound the optimization, discretization, or modelling errors on the right hand side of Theorem \ref{trivial-risk-bound-theorem}. 

In typical machine learning problems, the distribution $d\mu$ in \eqref{risk-definition} is unknown and we can only interact with it by drawing i.i.d. samples from $d\mu$. The (true) risk \eqref{risk-definition} is then approximated by the empirical risk
\begin{equation}\label{empirical-risk}
    \mathcal{R}_{N'}(u^*) = \frac{1}{N'}\sum_{i=1}^{N'}l(x_i,u^*(x_i),z_i),
\end{equation}
where $(x_i,z_i)_{i=1}^{N'}$ are i.i.d. samples from $d\mu$ constituting the test dataset. This is akin to a Monte Carlo discretization of the integral in \eqref{risk-definition}. For applications in numerical PDEs, however, the distribution $d\mu$ is typically known explicitly. In these cases, the integral in \eqref{risk-definition} can potentially be more effectively discretized as
\begin{equation}\label{quadrature-risk-minimization}
    \mathcal{R}_{N'}(u^*) = \sum_{i=1}^{N'}w_il(x_i,u^*(x_i),z_i),
\end{equation}
where the $\omega_i$ are quadrature weights and the $(x_i,z_i)$ are quadrature points. The weights and points can be taken to be accurate to a given high order or may be determined via quasi-Monte Carlo integration methods \cite{longo2021higher}, for instance.

To ensure that we are accurately estimating the true risk of the function $u^*$ we need to obtain a bound on the discretization error
\begin{equation}
    |\mathcal{R}_{N'}(u^*) - \mathcal{R}(u^*)|.
\end{equation}
As an example, in the case of the classification loss we can apply Hoeffding's inequality \cite{hoeffding1994probability} to obtain
\begin{equation}\label{basic-hoeffding-bound}
    \mathbb{P}(|\mathcal{R}_{N'}(u^*) - \mathcal{R}(u^*)| \geq \epsilon) \leq 2\exp(-2\epsilon^2 / N'),
\end{equation}
since the loss $l$ is bounded between $0$ and $1$. This implies that with a large number of samples $N'$, we can estimate the classification error probability of a given \textit{fixed} model $u^*$ to accuracy $O((N')^{-\frac{1}{2}})$ with high probability.

We remark that in order for this approach to be rigorously correct, the test dataset used to evaluate $\mathcal{R}_N(u^*)$ must be \textit{independent} of the function $u^*$. This means for instance that the procedure used to determine $u^*$ cannot depend upon the test accuracy (using the same test dataset) $\mathcal{R}_N(u')$ of \textit{any other model} $u'$, i.e. it cannot depend upon previously published results run on the same test dataset. Of course, in practice this is violated in the deep learning community due to the expense of obtaining datasets and the consequent necessity of reusing test datasets many times for different models. Nonetheless, deviating from the ideal of redrawing a new test dataset for each model has been shown empirically to result in models that exhibit a significant drop in accuracy on new data \cite{recht2018cifar}. 

The disadvantage of an a posteriori bound is that if the estimated function $u^*$ does not have small risk, there is no way to fix this other than to try again with a different method for estimating $u^*$ (i.e. ``to tweak the hyperparameters of the method") and hope for the best, since we do not know why the method is failing. This is why we are trying to solve PDEs using neural networks in a way which allows \textit{a priori} error estimates to be obtained as described in Section \ref{a-priori-section}. This will allow us to obtain error bounds before applying our method, and further, if the method does not work, it allows us to conclude that the model class $\mathcal{F}_\Theta$ cannot accurately approximate the PDE solution.

We also note that the simple test error bound derived in \eqref{basic-hoeffding-bound} relied critically upon the fact that the loss function is bounded (in the case of classification). Unfortunately, the loss functions used to solve PDEs are typically not bounded. This means that the test error cannot be used to bound the generalization error in this context. The reason is that one would have to know how smooth the neural network function is in order to use quadrature which guarantees a certain error. Such bounds on the derivative norms of trained neural networks are not available to the best of our knowledge.  In our method, the fact that we can control the complexity of the numerical solution, see Section \ref{sec:uniform_error}, implies that we can obtain bounds on the true (i.e. continuous) energy and on the true residual in the case of the variational formulation and PINNs loss, respectively. This enables us to calculate a posteriori estimates on the energy and on the residuals using a new test dataset (or a new set of quadrature points). This permits the calculation of reference solutions even when the true solution is not known, which to the best of our knowledge cannot be done with other neural networks based methods.

\section{Shallow Neural Network Model Classes}\label{sec:model-class}
In this section, we introduce the model class $\mathcal{F}_\Theta$ over which we will optimize the empirical loss \eqref{eq-593}. This classical choice is to take $\mathcal{F}_\Theta$ to be an $n$-dimensional subspace of an appropriate Sobolev space. In our approach, we instead take $\mathcal{F}_\Theta$ to be non-linear expansions with respect to a suitable collection of functions $\mathbb{D}$, called a dictionary.

Specifically, for a set $\mathbb{D}\subset C^m(\Omega)$, we consider
\begin{equation}\label{NN}
\Sigma_{n,M}(\mathbb{D}) = \left\{\sum_{i=1}^n a_id_i,~d_i\in \mathbb{D},~\sum_{i=1}^n |a_i| \leq M\right\}.
\end{equation}
Note that here we restrict the $\ell^1$-norm of the coefficients $a_i$ in the expansion. In addition, we take our dictionary $\mathbb{D}\subset C^m(\Omega)$ (instead of $H^m(\Omega)$ since we will discretize the resulting integrals using quadrature point evaluations). In some cases, we will also need to consider the set
\begin{equation}
    \Sigma_{n,\infty}(\mathbb{D}) = \left\{\sum_{i=1}^n a_id_i,~d_i\in \mathbb{D}\right\}
\end{equation}
with no restriction on the coefficients. We then take the model class $\mathcal{F}_\Theta$ to be
\begin{equation}
    \mathcal{F}_{n,M} = \Sigma_{n,M}(\mathbb{D})
\end{equation}
which is parameterized by $\Theta = (n,M)$. Here the dependence on $\mathbb{D}$ is suppressed since the dictionary $\mathbb{D}$ will typically be fixed throughout our analysis. 

For shallow neural networks with ReLU$^k$ activation function $\sigma = \max(0,x)^k$ the dictionary $\mathbb{D}$ would be taken as \cite{siegel2021optimal}
\begin{equation}
 \mathbb{D} = \mathbb{P}_k^d := \{\sigma_k(\omega\cdot x + b):~\omega\in S^{d-1},~b\in [c_1,c_2]\}\subset L^2(B_1^d),
\end{equation}
where $S^{d-1} = \{\omega\in \mathbb{R}^d:~|\omega| = 1\}$ is the unit sphere. Here $c_1$ and $c_2$ are chosen to satisfy
\begin{equation}
 c_1 < \inf \{x\cdot \omega:x\in \Omega, \omega\in S^{d-1}\} < \sup\{x\cdot \omega:x\in \Omega, \omega\in S^{d-1}\}< c_2.
\end{equation}
The default choice $[c_{1},c_{2}] = [-2,2]$ is used in our experiments in Section 8 for the cases $\Omega \subset B_{1}^{d}$, where $B_{1}^{d}$ is the closed $d$-dimensional unit ball. We note that $\mathbb{P}_k^d\subset C^m(\Omega)$ whenever $k > m$ and that in this case $\displaystyle |\mathbb{P}_k^d| = \sup_{g\in \mathbb{P}_k^d}\|g\|_{H^m(\Omega)} < \infty$. In this case the model class would be given by
\begin{equation}
    \mathcal{F}_{n,M} = \Sigma_{n,M}(\mathbb{P}_k^d) = \left\{\sum_{i=1}^n a_i\sigma_k(\omega_i\cdot x + b_i),~\omega_i\in S^{d-1},~b_i\in [c_1,c_2],~\sum_{i=1}^n |a_i| \leq M\right\},
\end{equation}
which is the class of shallow ReLU$^k$ neural networks with width $n$ and coefficients bounded in $\ell^1$ by $M$.

In the case of a general activation function $\sigma$, the corresponding dictionary is given by
\begin{equation}
    \mathbb{D}_\sigma = \left\{\sigma(\omega\cdot x + b):~(\omega,b)\in \Theta\right\},
\end{equation}
where $\Theta\subset \mathbb{R}^d\times \mathbb{R}$ is compact. In this case, we have $\mathbb{D}_\sigma\subset C^m(\Omega)$ and $|\mathbb{D}_\sigma| < \infty$ whenever $\sigma\in C^m(\Omega)$. In this case, the function class would consist of
\begin{equation}
    \mathcal{F}_{n,M} = \Sigma_{n,M}(\mathbb{D}_\sigma) = \left\{\sum_{i=1}^n a_i\sigma(\omega_i\cdot x + b_i),~(\omega_i,b_i)\in \Theta,~\sum_{i=1}^n |a_i| \leq M\right\},
\end{equation}
which is the class of shallow neural networks with activation function $\sigma$, bounded inner coefficients and outer coefficients bounded in $\ell^1$ by $M$.

\subsection{Barron space regularity}\label{sec:barron-reg}
In this section, we introduce the notion of regularity which corresponds to the model class of shallow neural networks introduced in Section \ref{sec:model-class}. As in Section \ref{sec:model-class}, we give this notions in the abstract setting of a general dictionary $\mathbb{D}\subset C^m(\Omega)$. 

Consider the closed convex hull of $\mathbb{D}$, defined by
\begin{equation}
 B_1(\mathbb{D}) := \overline{\bigcup_{n=1}^\infty \Sigma_{n,1}(\mathbb{D})},
\end{equation}
where $\Sigma_{n,1}$ is defined in \eqref{NN}.
Note that here the closure is taken in $H^m(\Omega)$. Associated with the convex set $B_1(\mathbb{D})$, we define the gauge norm (also called the Minkowski functional) by
\begin{equation}
    \|f\|_{\mathcal{K}_{1}(\mathbb{D})}=\inf \left\{c>0: f \in c B_{1}(\mathbb{D})\right\}.
\end{equation}
The norm $\|\cdot\|_{\mathcal{K}_1(\mathbb{D})}$, which is also called the variation norm corresponding to the dictionary $\mathbb{D}$, is constructed precisely so that $B_1(\mathbb{D})$ its unit ball. We further define the function space
\begin{equation}
    \mathcal{K}_{1}(\mathbb{D}):=\left\{f \in H^m(\Omega):\|f\|_{\mathcal{K}_{1}(\mathbb{D})}<\infty\right\}.
\end{equation}
Important fundamental properties of this space, for instance is the fact that if $\mathbb{D}$ is a uniformly bounded dictionary, i.e. if $\sup_{d \in \mathbb{D}}\|d\|_{H}=K_{\mathbb{D}}<\infty$, then the space $\mathcal{K}_1(\mathbb{D})$ is a Banach space, can be found in \cite{siegel2021characterization}.

The utility of the space $\mathcal{K}_1(\mathbb{D})$ is due to the fact that its elements can be efficiently approximated by non-linear dictionary expansions. In particular, the following classical bound holds \cite{barron1993universal,pisier1981remarques}
\begin{equation}\label{maurey-rate}
    \inf_{f_n\in \Sigma_{n,M}(\mathbb{D})} \|f - f_n\|_{H^m(\Omega)} \leq |\mathbb{D}|\|f\|_{\mathcal{K}_1(\mathbb{D})}n^{-\frac{1}{2}},
\end{equation}
for $M = \|f\|_{\mathcal{K}_1(\mathbb{D})}$. Because of this approximation result, we consider regularity assumptions with respect to the $\mathcal{K}_1(\mathbb{D})$-norm, i.e. we assume that the variation norm of the PDE solution can be controlled. For the specific variation spaces corresponding to the dictionaries $\mathbb{P}_k^d$, such regularity results for a variety of PDEs have been obtained \cite{CiCP-28-1707,lu2021priori,chen2022regularity,chen2021representation}.

Recently, the spaces $\mathcal{K}_1(\mathbb{P}_k^d)$ for the dictionaries $\mathbb{P}_{k}^{d}$ corresponding to shallow ReLU$^k$ neural networks have been characterized in terms of the Radon transform \cite{siegel2021characterization,parhi2020banach,ongie2019function,parhi2021kinds} and they are closely related to the Ridgelet spaces \cite{candes1998ridgelets}. 
In addition, precise approximation theoretic properties of the space $\mathcal{K}_1(\mathbb{P}_k^d)$, such as the asymptotics of its metric entropy and $n$-widths can be found in \cite{siegel2021sharp}. In \cite{siegel2021sharp} it is also shown that the approximation rate \eqref{maurey-rate} can be improved to
\begin{equation}\label{sharp-rate}
    \inf_{f_n\in \Sigma_{n,M}(\mathbb{P}_k^d)} \|f - f_n\|_{H^m(\Omega)} \lesssim \|f\|_{\mathcal{K}_1(\mathbb{P}_k^d)}n^{-\frac{1}{2}-\frac{2k+1}{2d}},
\end{equation}
with $M\lesssim \|f\|_{\mathcal{K}_1(\mathbb{P}_k^d)}$for the dictionary $\mathbb{D} = \mathbb{P}_k^d$. Similar results for more general activation functions can be found in \cite{siegel2020approximation}. Pointwise properties of functions in $\mathcal{K}_1(\mathbb{P}_1^d)$, which is also called the Barron space \cite{ma2022barron}, have also been obtained in \cite{wojtowytsch2022representation}.

\section{Greedy Algorithms}\label{sec:greedy-algorithms}
In this section, we address the problem of bounding the optimization error in Theorem \ref{trivial-risk-bound-theorem} when optimizing the empirical loss over the model class $\mathcal{F}_\Theta = \Sigma_{n,M}(\mathbb{D})$ introduced in Section \ref{sec:model-class}. For simplicity, we denote the numerical solution $\bar{ u}_{\Theta,N} = \bar{u}_{n,M,N}$ as ${u}_{n}$  in this section.

As in Section \ref{sec:model-class}, let $\mathbb{D} \subset H$ be a dictionary in Hilbert space $H$ (in our applications typically $H = H^m(\Omega)$ for some domain $\Omega$). Greedy algorithms for expanding a function $u\in H$ as a linear combination of the dictionary elements $\mathbb{D}$ are fundamental in approximation theory \cite{devore1996some,temlyakov2008greedy,temlyakov2011greedy} and signal processing \cite{mallat1993matching,pati1993orthogonal}. Greedy methods have also been proposed for optimizing shallow neural networks \cite{lee1996efficient,dereventsov2019greedy} and for solving PDEs numerically \cite{figueroa2012greedy,cances2013greedy,ammar2006new,le2009results}.

The class $\mathcal{K}_1(\mathbb{D})$ which was introduced in Section \ref{sec:barron-reg} is a natural target space in the analysis of greedy algorithms \cite{temlyakov2008greedy,temlyakov2011greedy}. Given the dictionary $\mathbb{D}$ and a target function $u$ or a convex loss function $\mathcal{L}$, greedy algorithms either approximate $f$ or approximately minimize $\mathcal{L}$ by a finite linear combination of dictionary elements:
\begin{equation}
   u_{n} = \sum_{i=1}^n a_ig_i,
\end{equation}
with $g_i\in \mathbb{D}$. 
The two types of greedy algorithm we discuss here are the relaxed greedy algorithm (RGA) and orthogonal greedy algorithm (OGA).

\subsection{Relaxed greedy algorithm}
We consider the following version of the RGA, which explicitly optimizes $\mathcal{L}$ over the convex hull of the dictionary,
\begin{equation}\label{relaxed-greedy}
    u_0 = 0,~g_n = \arg\max_{g\in \mathbb{D}}\langle g, \nabla\mathcal{L}(u_{n-1})\rangle_H,~u_n = (1-\alpha_n)u_{n-1} - M\alpha_n g_n.
\end{equation}
Here the dictionary $\mathbb{D}$ is assumed to symmetric (i.e. $g\in \mathbb{D}$ implies that $-g\in \mathbb{D}$ as well), the sequence $\alpha_n$ is given by $\alpha_n = \min\left(1,\frac{2}{n}\right)$, and $M$ is a regularization parameter which controls the $\mathcal{K}_1(\mathbb{D})$-norm of the iterates $u_n$. This algorithm was first introduced and analyzed by Jones \cite{jones1992simple} for function approximation (i.e. $\mathcal{L}(u) = \|u - f\|_H^2$), and has been extended to the optimization of general convex objectives as well \cite{zhang2003sequential}. The convergence theorem we will use in our analysis, which is closely related to Theorem IV.2 in \cite{zhang2003sequential}, is the following.
\begin{theorem}\label{RGA_convergence}
Suppose that the dictionary $\mathbb{D}$ is symmetric and satisfies $\sup _{d \in \mathbb{D}}\|d\|_{H} \leq C<\infty .$ Let the iterates $u_{n}$ be given by the RGA \eqref{relaxed-greedy}. Assume that the loss function $\mathcal{L}$ is convex and $K$-smooth (on the Hilbert space $H).$ Recall that $K$-smoothness means that for any $u,v\in H$ we have
\begin{equation}
    \mathcal{L}(u) \leq \mathcal{L}(v) + \langle\nabla \mathcal{L}(v),u-v\rangle_H + \frac{K}{2}\|u-v\|^2_H.
\end{equation}
Then we have $u_n\in \Sigma_{n,M}$ and
\begin{equation}\label{conv_rga}
\mathcal{L}\left(u_{n}\right)-\underset{\|v\|_{\mathcal{K}_{1}(\mathbb{D})} \leq M}{\inf } \mathcal{L}(v) \leq \frac{32(C M)^{2} K}{n}
\end{equation}
\end{theorem}
This theorem will be applied in the case $\mathcal{L} = \mathcal{R}_N$ is the empirical risk to bound the \textit{optimization error}. In particular it yields that
\begin{equation}\label{rate-RN}
    \mathcal{R}_N\left(u_{n}\right) -\underset{v\in \Sigma_{n,M}(\mathbb{D})}{\inf} \mathcal{R}_N(v) \leq  \mathcal{R}_N\left(u_{n}\right)-\underset{\|v\|_{\mathcal{K}_{1}(\mathbb{D})}}{\inf} \mathcal{R}_N(v) \lesssim n^{-1}.
\end{equation}

We remark that this theorem holds for any convex and $K$-smooth loss function. This means that the RGA can be applied to non-linear equations in addition to the linear equations introduced in Section \ref{basic-section}, provided that the non-linear equations admit a variational formulation with a convex energy function.

\begin{proof}
 Since $u_0 = 0$ and $u_k$ is a convex combination of $u_{k-1}$ and $-Mg_k$, we see by induction that $u_k\in \Sigma_{k,M}$.
 The $K$-smoothness of the objective $\mathcal{L}$ implies that
 \begin{equation}
  L(u_k) \leq L(u_{k-1}) + \langle \nabla \mathcal{L}(u_{k-1}), u_k - u_{k-1}\rangle + \frac{K}{2}\|u_k - u_{k-1}\|_H^2.
 \end{equation}
 Using the iteration \eqref{relaxed-greedy}, we see that $u_k - u_{k-1} = -s_ku_{k-1}-Ms_kg_k$. Plugging this into the above equation, we get
 \begin{equation}
  \mathcal{L}(u_k) \leq \mathcal{L}(u_{k-1}) - s_k\langle \nabla \mathcal{L}(u_{k-1}), u_{k-1} + Mg_k\rangle + \frac{Ks_k^2}{2}\|u_{k-1} + Mg_k\|_H^2.
 \end{equation}
 Since the dictionary elements $g_k$ satisfy $\|g_k\|_H \leq C$ and $\|u_{k-1}\|_{\mathcal{K}_1(\mathbb{D})} \leq M$, we see that $\|u_{k-1}\|_H \leq CM$ as well. Plugging this into the previous equation implies the bound
 \begin{equation}\label{eq-479}
  \mathcal{L}(u_k) \leq \mathcal{L}(u_{k-1}) - s_k\langle \nabla \mathcal{L}(u_{k-1}), u_{k-1} + Mg_k\rangle + 2(CM)^2Ks_k^2.
 \end{equation}
 Now let $z$ with $\|z\|_{\mathcal{K}_1(\mathbb{D})} \leq M$ be arbitrary. Then also $\|-z\|_{\mathcal{K}_1(\mathbb{D})} \leq M$ and the $\arg\max$ characterization of $g_k$ \eqref{argmax-approximation}  implies that
 \begin{equation}
  \langle \nabla \mathcal{L}(u_{k-1}), -z\rangle \leq \langle \nabla \mathcal{L}(u_{k-1}), Mg_k\rangle.
 \end{equation}
 Using this in equation \eqref{eq-479} gives
 \begin{equation}
  \mathcal{L}(u_k) \leq \mathcal{L}(u_{k-1}) - s_n\langle \nabla \mathcal{L}(u_{k-1}), u_{k-1} - z\rangle + 2(CM)^2Ks_k^2.
 \end{equation}
 The convexity of $\mathcal{L}$ means that $\mathcal{L}(u_{k-1}) - \mathcal{L}(z) \leq \langle \nabla \mathcal{L}(u_{k-1}), u_{k-1} - z\rangle$. Using this and subtracting $\mathcal{L}(z)$ from both sides of the above equation gives
 \begin{equation}
  \mathcal{L}(u_k) - \mathcal{L}(z) \leq (1 - s_k)(\mathcal{L}(u_{k-1}) - \mathcal{L}(z)) + 2(CM)^2Ks_k^2.
 \end{equation}
 Expanding the above recursion (using that $s_k\leq 1$), we get that
 \begin{equation}
  \mathcal{L}(u_n) - \mathcal{L}(z) \leq \left(\prod_{k=1}^n(1-s_k)\right)(\mathcal{L}(u_0) - \mathcal{L}(z)) + 2(CM)^2K\sum_{i=1}^n\left(\prod_{k=i+1}^n(1-s_k)\right)s_i^2.
 \end{equation}
 Using the choice $s_k = \max\left(1,\frac{2}{k}\right)$, for which $s_1 = 1$, we get
 \begin{equation}\label{eq-268}
  \mathcal{L}(u_n) - \mathcal{L}(z) \leq 2(CM)^2K\sum_{i=1}^n\left(\prod_{k=i+1}^n(1-s_k)\right)s_i^2.
 \end{equation}
 Finally, we bound the product $\displaystyle\prod_{k=i+1}^n(1-s_k)$ using that $\log(1+x) \leq x$ as
 \begin{equation}
  \log\left(\prod_{k=i+1}^n(1-s_k)\right) \leq -\sum_{k=i+1}^ns_k = -\sum_{k=i+1}^n\frac{2}{k} \leq -\int_{i+1}^{n+1}\frac{2}{x}dx \leq 2(\log(i+1) - \log(n+1)),
 \end{equation}
 for $i \geq 1$. Thus, $\prod_{k=i+1}^n(1-s_k) \leq \frac{(i+1)^2}{(n+1)^2}$. Using this in equation \eqref{eq-268}, we get
 \begin{equation}
  \mathcal{L}(u_n) - \mathcal{L}(z) \leq 2(CM)^2K\sum_{i=1}^n \frac{(i+1)^2}{(n+1)^2}s_i^2 \leq 8(CM)^2K\frac{1}{(n+1)^2}\sum_{i=1}^n \frac{(i+1)^2}{i^2}.
 \end{equation}
 Crudely bounding $\frac{(i+1)^2}{i^2} \leq 4$ for $i \geq 1$, we get
\begin{equation}
  \mathcal{L}(u_n) - \mathcal{L}(z) \leq 32(CM)^2K\frac{n}{(n+1)^2} \leq \frac{32(CM)^2K}{n},
 \end{equation}
 Taking the infimum over $z$ with $\|z\|_{\mathcal{K}_1(\mathbb{D})} \leq M$ gives the result.
\end{proof}

\subsection{The Orthogonal Greedy Algorithm}\label{OGA}
The OGA only applies to function approximation, not to general convex optimization, and is given by
\begin{equation}\label{orthogonal-greedy}
    u_0 = 0,~g_n = \arg\max_{g\in \mathbb{D}}|\langle g, u_{n-1} - u\rangle_H|,~u_n = P_n(u),
\end{equation}
where $P_n$ is the orthogonal projection onto the span of $g_1,...,g_n$. Note here that the residual $u_{n-1} - u$ is the gradient $\nabla \mathcal{L}(u_{n-1})$ for the quadratic function $\mathcal{L}(u_{n-1}) = \frac{1}{2}\|u_{n-1} - u\|^2_H$. We remark that since this algorithm only applies to function approximation in a Hilbert space, our methods based upon the OGA can only be used to solve linear PDEs.

This algorithm was first analyzed in \cite{devore1996some}, where an $O(n^{-\frac{1}{2}})$ convergence rate is derived. Recently, it has been shown that this convergence rate can be significantly improved for the dictionaries whose convex hull $B_1(\mathbb{D})$ has small entropy \cite{siegel2022optimal}. In this section, we explain how to use the orthogonal greedy algorithm to solve linear PDEs and analyze the optimization error this induces.

When solving linear PDEs, the discretized energy function $\mathcal{R}_N(v)$ (or $\mathcal{R}_{n,\delta}(v)$) defined in \eqref{eq-593} or \eqref{eq-507} is a quadratic function of $v$. In particular, we have
\begin{equation}\label{eq-743}
    \mathcal{R}_N(v) = \frac{1}{2}\left(\sum_{i=1}^{N}w_ia_0(x_i)v(x_i)^2 + \sum_{i=1}^N\sum_{|\alpha| = m}w_ia_\alpha(x_i)(\partial^\alpha v(x_i))^2\right) - \sum_{i=1}^Nw_i f(x_i)v(x_i),
\end{equation}
with Neumann boundary conditions and an analogous expression with Dirichlet boundary conditions. In the following we assume that the quadrature weights $w_i > 0$. Then the loss $\mathcal{L} = \mathcal{R}_N(v)$ is equivalent to
\begin{equation}\label{oga-quadratic-equation}
    \mathcal{R}_N(v) = \|I_{m,N}(v) - u_{N}\|^2_{a,N},
\end{equation}
where the evaluation map $I_{m,N}:C^m(\Omega)\rightarrow \mathbb{R}^{P}$ is given by evaluating the function $v$ and all derivatives of order $m$ at the quadrature points $x_i$. Specifically, this map is given by
\begin{equation}
    \left(I_{m,N}(v)\right)_{\alpha,i} = \partial^\alpha v(x_i),
\end{equation}
where the index set $(\alpha,i)$ runs over all multi-indices $\alpha$ such that either $|\alpha| = 0$ (the terms with no derivatives) or $|\alpha| = m$ and indices $i=1,...,N$. Consequently $P=N\times N\binom{m+d-1}{d-1}$. The norm $\|\cdot\|_{a,N}$ on $\mathbb{R}^{P}$ is given by the weighted norm
\begin{equation}\label{discrete-energy-norm}
    \|x\|^2_{a,N} = \sum_{(\alpha,i)}w_ia_{\alpha}(x_i)x_{(\alpha,i)}^2.
\end{equation}
Finally, $u_N$ is the minimizer of the quadratic \eqref{eq-743} in $\mathbb{R}^{P}$. Specifically, the components of $u_N$ are given by
\begin{equation}
    (u_N)_{(\alpha,i)} = \begin{cases}
    0 & |\alpha| = m\\
    a_0(x_i)^{-1}f(x_i) & |\alpha| = 0.
    \end{cases}
\end{equation}
Crucially, $u_N$ can be determined solely from knowledge of the right hand side $f$ and the coefficients $a_0$ and does not require knowledge of the true solution $u$.

Using the orthogonal greedy algorithm to minimize the quadratic objective \eqref{oga-quadratic-equation} results in the iteration
\begin{equation}\label{linear-PDE-OGA}
    u_{0,N} = 0,~g_n = \arg\max_{g\in \mathbb{D}}|\langle I_{m,N}(g), u_{n-1,N} - u_N\rangle_{a,N}|,~u_{n,N} = P_n(u_N),
\end{equation}
where the projection $P_n$ is onto the span of the elements $I_{m,N}(g_1),...,I_{m,N}(g_n)$ with respect to the norm $\|\cdot\|_{a,N}$ on $\mathbb{R}^{P}$.

In a similar manner the PINNs risk \eqref{PINN-risk} can be handled using the orthogonal greedy algorithm as long as the equation is linear. In this case the discretized risk is given by
\begin{equation}
    \mathcal{R}_N(v) = \frac{1}{2}\left(\sum_{i=1}^{N}w_i\left(a_0(x_i)v(x_i) + \sum_{|\alpha|=m}a_\alpha(x_i)\partial^\alpha v(x_i) - f(x_i)\right)^2\right).
\end{equation}
This defines a quadratic function, and thus an inner product (possibly with kernel) on $\mathbb{R}^P$. We then maximize and project with respect to this inner product and the dictionary is embedded into $\mathbb{R}^P$ via the map $I_{m,N}$, resulting in an analogous method to \eqref{linear-PDE-OGA}. Using the PINN risk allows us to tackle non-symmetric linear problems which may not have a variational formulation.

The estimation of the optimization error follows from the estimates derived in \cite{siegel2022optimal}. In particular, we quote the following theorem. Note that this theorem gives an upper bound and for certain dictionaries it is possible that the convergence rate of the OGA may be even faster.
\begin{theorem}\label{OGA-convergence-theorem}
    Let $H$ be a Hilbert space and $\mathbb{D}\subset H$ a dictionary such that the metric entropy of the convex hull of $\mathbb{D}$ satisfies
    \begin{equation}
        \epsilon_n(B_1(\mathbb{D}))_H \leq Cn^{-\frac{1}{2}-\gamma}
    \end{equation}
    for some $\gamma > 0$. Then for any $v\in \mathcal{K}_1(\mathbb{D})$, we have
    \begin{equation}
        \|u_n - u\|^2_H \leq \|v-u\|_H^2 + K\|v\|_{\mathcal{K}_1(\mathbb{D})}^2n^{-1-2\gamma},
    \end{equation}
    where $K$ is a constant only depending upon $C$ and $\gamma$.
\end{theorem}
Here the metric entropy $\epsilon_n(B_1(\mathbb{D}))_H$ is a measure of compactness of the set $B_1(\mathbb{D})$ with respect to the norm of $H$. For a precise definition and development of its properties, see for instance \cite{lorentz1996constructive}, Chapter 15. The important point is that the dictionary $\mathbb{P}_k^d$ satisfies \cite{siegel2021sharp}
\begin{equation}
    \epsilon_n(B_1(\mathbb{P}_k^d))_{H^m(\Omega)} \lesssim n^{-\frac{1}{2}-\frac{2(k-m)+1}{2d}}.
\end{equation}
Thus, for this dictionary the value of $\gamma$ in Theorem \ref{OGA-convergence-theorem} is $\gamma = \frac{2(k-m) +1}{2d}$. When solving the discrete equation \eqref{oga-quadratic-equation} using the orthogonal greedy algorithm it is important to note that up to logarithmic factors this entropy bound also holds in the $C^m(\Omega)$-norm when $k = m+1$ \cite{bach2017breaking,siegel2021sharp}. It is conjectured but not yet proven that this also holds for larger values of $k$. This means that since the evaluation map $I_{m,N}:C^m(\Omega)\rightarrow \mathbb{R}^{P}$ is bounded uniformly in $N$ we have
\begin{equation}
    \epsilon_n(I_{m,N}(B_1(\mathbb{P}_k^d)))_{a,N} \leq Cn^{-\frac{1}{2}-\gamma}
\end{equation}
holds uniformly in $N$ for $\gamma = \frac{2(k-m) +1}{2d}$. Hence, denoting by $\bar{u}_{n,N}\in \Sigma_{n,\infty}(\mathbb{P}_k^d)$ the solution produced by the OGA at step $n$, we have for any $M$ that
\begin{equation}\label{conv_oga_1}
    \mathcal{R}_N(\bar{u}_{n,N}) - \inf_{v\in \Sigma_{n,M}(\mathbb{P}_k^d)}\mathcal{R}_N(v) \leq \mathcal{R}_N(\bar{u}_{n,N}) - \inf_{\|v\|_{\mathcal{K}_1(\mathbb{P}_k^d) \leq M}}\mathcal{R}_N(v) \lesssim n^{-1-\frac{2(k-m)+1}{d}}.
\end{equation}
This follows by taking the infimum over $\|v\|_{\mathcal{K}_1(\mathbb{P}_k^d)} \leq M$ in the conclusion of Theorem \ref{OGA-convergence-theorem}. This is precisely the optimization error bound we desire. Of course, this analysis applies to more general dictionaries $\mathbb{D}$ as well, provided that the metric entropy $\epsilon_n(\mathbb{D})$ can be estimated.

Although the OGA attains the best convergence rate of the greedy algorithms, it is also the most computationally expensive since it requires an orthogonal projection at every step. In addition, it can only be applied to function approximation, which corresponds in our case to linear PDEs. A final drawback of the OGA is that the $\mathcal{K}_1(\mathbb{D})$-norm of the iterates cannot be a priori bounded for general dictionaries as shown in \cite{siegel2022optimal}. As a result, we can only have a priori guarantee that the numerical solution satisfies $\bar{u}_{n,N}\in \Sigma_{n,\infty}(\mathbb{D})$. This means that our a priori generalization analysis only holds when using the RGA to optimize the empirical loss. Despite this, we have empirically observed the improved convergence rate of the OGA a posteriori and it significantly outperforms the RGA in our experiments. 

\section{Solving the argmax sub-problem}
In order to implement the relaxed and orthogonal greedy algorithms, we need to be able to numerically solve the substep
\begin{equation}\label{pga-oga-argmax}
g_n = \arg\max_{g\in \mathbb{D}} |\langle g, \nabla\mathcal{L}(u_{n-1})\rangle|.
\end{equation}
In fact, for the convergence analysis it is sufficient that the argmax in \eqref{pga-oga-argmax} is not solved exactly, but rather is approximated in the following sense
\begin{equation}\label{argmax-approximation}
 |\langle g_n, \nabla\mathcal{L}(u_{n-1})\rangle| \geq \frac{1}{R}\max_{g\in \mathbb{D}} |\langle g, \nabla\mathcal{L}(u_{n-1})\rangle|
 \end{equation}
for some fixed $R > 1$. This is a more tractable problem for most dictionaries. 

\subsection{Exactly solving the argmax sub-problem}
We remark that in low dimensions and for certain dictionaries the argmax subproblem can be efficiently solved exactly. This is due to the fact that the objective
\begin{equation}
    |\langle g, \nabla\mathcal{L}(u_{n-1})\rangle| = \left|\sum_{i=1}^N \sum_{|\alpha|=m} \left( a_{\alpha} \partial^{\alpha} u_{n-1}(x_i), \partial^{\alpha} g(x_i) \right) + \left(a_0 u_{n-1}(x_i) - f(x_i), g(x_i)\right)\right|
\end{equation}
is really a sum over a finite number $N$ of quadrature points. In this case the set of possible hyperplane partitions of the quadrature points $x_i$ can be enumerated and this can be used to exactly determine the $\arg\max$ in \eqref{pga-oga-argmax}. This algorithms is unfortunately intractable in high dimensions since its complexity scales as $O(N^d\log(N))$ \cite{bartlett2002rademacher,siegel2022optimal}. As a result, for higher dimensional problems we must resort to heuristics to approximate the subproblem \eqref{pga-oga-argmax} or consider different dictionaries for which this problem can be solved more efficiently. In our high dimensional numerical experiments, we use a special dictionary for which this argmax can be efficiently solved.

\subsection{Numerical approximation of the argmax sub-problem}
Next we describe the numerical heuristics we use in our experiments to approximately solve the argmax sub-problem in \eqref{pga-oga-argmax}.
Note that here the inner product in \eqref{pga-oga-argmax} is the energy inner product associated with the elliptic PDE we are solving.
Our first step is to make the target function $|\langle g,\nabla \mathcal{L}(u_{n-1})\rangle|$ differentiable, so we instead consider the following equivalent optimization problem:
\begin{equation}\label{pga-oga-rga-argmin}
g_n = \arg\min_{g\in \mathbb{D}} -\dfrac{1}{2}\langle g, \nabla\mathcal{L}(u_{n-1})\rangle^2,
\end{equation}
where
\begin{equation}
    \langle \sigma(\omega\cdot x+b), \nabla \mathcal{L}(u_{n-1}) \rangle = \sum_{|\alpha|=m} \left( a_{\alpha} \partial^{\alpha} u_{n-1}, \partial^{\alpha} \sigma(\omega\cdot x+b) \right) + \left( a_0 u_{n-1} - f, \sigma(\omega\cdot x+b) \right).
\end{equation}
\begin{comment}
$\omega \in \mathbb{R}^d$ and $b\in \mathbb{R}$ when we solve PDEs \eqref{model-problem}. 
\end{comment}
Here  we choose the dictionary $\mathbb{D}\subset\mathbb{R}^d$ as $\mathbb{D} = \mathbb{P}_{k}^{d}$, which is naturally parameterized by $\omega\in S^{d-1}$ and $b\in [-c,c]$ \cite{siegel2021characterization}. We also enforce the constraint $\|\omega\|=1$ by taking $\omega=\pm 1$ for 1D case and $\omega=(cos\theta,\sin\theta)$  based on the polar coordinates for 2D case.
 The low-dimensional optimization problem in (\ref{pga-oga-rga-argmin}) is typically non-convex so it may be very difficult to obtain the global minimum. Our approach is to obtain a good initial guess by choosing many samples initially on $\omega-b$ parameter space and evaluating the objective function at each of them. More specifically, we sample $b_i=-c+\frac{2ci}{N_b},(i=0,\cdots, N_b)$, $w_0=-1$, $w_1=1$ (1D case), and $\theta_j=\frac{2\pi j}{N_\theta},~(j=0,\cdots,N_\theta)$ (2D case)
to find the best initial samples  by evaluating (\ref{pga-oga-rga-argmin}) at each $(b_i,w_j)$.
 We then further optimize the best initial sample points using gradient descent or Newton's method. 
For the RGA, we optimize $g_n = \displaystyle\arg\min_{g\in \mathbb{D}} -\langle g, \nabla\mathcal{L}(u_{n-1})\rangle$ instead of \eqref{pga-oga-argmax}.

\section{Uniform Error Bounds}\label{sec:uniform_error}

In this section, we explain how to bound the discretization error
\begin{equation}
    \sup_{f\in \mathcal{F}_\Theta} \left|\mathcal{R}_N(f) - \mathcal{R}(f)\right|
\end{equation}
in Theorem \ref{trivial-risk-bound-theorem} when solving elliptic PDEs. Recall that the loss function we consider in this work corresponds to the variational formulation of an elliptic PDE and is given in equation \eqref{elliptic-loss}. 
\subsection{Uniform Monte Carlo Error}
The tool which we use to analyze the discretization error when the Monte Carlo discretization in equation \eqref{eq-593} is used is the Rademacher complexity \cite{bartlett2002rademacher}.
Given a class of functions $\mathcal{F}:\Omega\rightarrow \mathbb{R}$, and a collection of sample points $x_1,...,x_N\in \Omega$, the empirical Rademacher complexity of $\mathcal{F}$ is defined by
\begin{equation}
   \tilde{R}_N(\mathcal{F}) = \mathbb{E}_{\xi_1,...,\xi_N}\left[\sup_{h\in \mathcal{F}}\frac{1}{N}\sum_{i=1}^N \xi_ih(x_i)\right],
\end{equation}
  where $\xi_1,...,\xi_n$ are Rademacher random variables, i.e. uniformly distributed signs. The Rademacher complexity is obtained by averaging over the samples $x_i$, which we take to be uniformly distributed over $\Omega$, i.e. we have
  \begin{equation}
   R_N(\mathcal{F}) = \mathbb{E}_{x_1,...,x_N\sim \mu}\mathbb{E}_{\xi_1,...,\xi_N}\left[\sup_{h\in \mathcal{F}}\frac{1}{N}\sum_{i=1}^N \xi_ih(x_i)\right],
  \end{equation}
 where $\mu$ is the uniform distribution on $\Omega$. For the mixed boundary value problem, we will also need the Rademacher complexity with respect to the uniform distribution on the boundary $\partial\Omega$, which we denote by $R_{\partial,N}(\mathcal{F})$.
 
 The utility of the Rademacher complexity is its role in giving a law of large numbers which is uniform over the class $\mathcal{F}$, detailed by the following theorem.
 \begin{theorem}\cite[Proposition 4.11]{wainwright2019high}\label{rademacher-theorem}
 Let $\mathcal{F}$ be a set of functions. Then
 \begin{equation}
  \mathbb{E}_{x_1,...,x_N\sim\mu}\sup_{h\in \mathcal{F}} \left|\frac{1}{N}\sum_{i=1}^N h(x_i) - \int h(x)d\mu\right| \leq 2R_N(\mathcal{F}).
 \end{equation}
\end{theorem}

In order to apply Theorem \ref{rademacher-theorem} to the solution of PDEs via the class of Barron functions, we need to estimate the Rademacher complexity $R_N(\mathcal{L}_M)$ of the model class
\begin{equation}
    \mathcal{L}_M = \{l(u,Du,...,D^mu):~u\in \mathcal{F}_{n,M}\},
\end{equation}
where the loss function $l$ is given in equation \eqref{elliptic-loss} and the model class $\mathcal{F}_{n,M}$ is described in section \ref{sec:model-class}. We remark that the Rademacher complexity of the Barron class $\mathcal{F}_{n,M}$ corresponding to shallow ReLU networks has been estimated in \cite{ma2022barron}, so the novelty of our contribution is to generalize these bounds to the class $\mathcal{L}_M$ obtained by composing with the loss function \eqref{elliptic-loss}.

For this we will utilize the following fundamental lemma.
\begin{lemma}\label{rademacher-lemma}
 Let $\mathcal{F}, \mathcal{S}$ be classes of functions on $\Omega$. Then the following bounds hold.
 \begin{itemize}
  \item $R_N(\conv(\mathcal{F}))= R_N(\mathcal{F})$.
   \item Define the set $\mathcal{F} + \mathcal{S} = \{h(x) + g(x):~h\in \mathcal{F},~g\in \mathcal{S}\}$. We have
  \begin{equation}
   R_N(\mathcal{F} + \mathcal{S}) = R_N(\mathcal{F}) + R_N(\mathcal{S}).
  \end{equation}
  \item Suppose that $\phi:\mathbb{R}\rightarrow \mathbb{R}$ is $L$-Lipschitz. Let $\phi\circ \mathcal{F} = \{\phi(h(x)):~h\in \mathcal{F}\}$. Then
  \begin{equation}
   R_N(\phi\circ \mathcal{F})\leq LR_N(\mathcal{F}).
  \end{equation}
   \item Suppose that $f:\Omega\rightarrow \mathbb{R}$ is a fixed function. Let
  $f\cdot\mathcal{F} = \{f(x)h(x):~h\in \mathcal{F}\}.$
  Then
  \begin{equation}
   R_N(f\cdot \mathcal{F})\leq \|f(x)\|_{L^\infty(\Omega)}R_N(\mathcal{F}).
  \end{equation}

 \end{itemize}

\end{lemma}
\begin{proof}
 The first, second, and third of these statements are well-known facts, see \cite[Lemma 26.7]{shalev2014understanding} for the first,  \cite[Page 56]{mohri2018foundations} for the second and \cite[Lemma 26.9]{shalev2014understanding} for the third, so we only prove the fourth. 
 
 Suppose that $\|f(x)\|_{L^\infty(\Omega)} \leq 1$, the general result follows by a scaling argument. Let $x_1,...,x_N\in \Omega$ and consider the empirical Rademacher complexity
 \begin{equation}
 \tilde{R}_N(f\cdot\mathcal{F}) = \mathbb{E}_{\xi_1,...,\xi_N}\left[\sup_{h\in \mathcal{F}}\frac{1}{N}\sum_{i=1}^N \xi_if(x_i)h(x_i)\right].
 \end{equation}
 We observe that the right-hand side of the above equation, being an average of a supremum of linear functions, is a convex function of $\vec{f} = (f(x_1),...,f(x_N))$. Consequently, its maximum must be achieved at the extreme points of the set $\{\vec{y}:~\|\vec{y}\|_\infty \leq 1\}$, which correspond to the points where each component is $\pm1$. Thus we only need to consider the case where $f(x_i) = \epsilon_i\in \{\pm 1\}$. But then
 \begin{equation}
  \mathbb{E}_{\xi_1,...,\xi_N}\left[\sup_{h\in \mathcal{F}}\frac{1}{N}\sum_{i=1}^N \xi_i\epsilon_ih(x_i)\right] = \mathbb{E}_{\xi_1,...,\xi_N}\left[\sup_{h\in \mathcal{F}}\frac{1}{N}\sum_{i=1}^N \xi_ih(x_i)\right] = \tilde{R}_N(\mathcal{F}),
 \end{equation}
 since the $\epsilon_i$ simply permute the choices of sign $\xi_i$ in the expectation. Taking an average over the sample points $x_1,...,x_N$ completes the proof.
\end{proof}

Utilizing this lemma, we prove the following bound on the Rademacher complexity of the set $\mathcal{L}_M$.

\begin{theorem}\label{main-rademacher-bound}
 Let $\mathbb{D}\subset C^k(\Omega)$ for $k\geq m$ be a dictionary. Suppose that $\|a_\alpha\|_{L^\infty(\Omega)}, \|a_0\|_{L^\infty(\Omega)}\leq K$ and $\sup_{d\in \mathbb{D}} \|d\|_{W^{m,\infty}} \leq C$. Then the Rademacher complexity of the set $\mathcal{L}_M$ is bounded by
 \begin{equation}
  R_N(\mathcal{L}_M) \leq CKM\sum_{|\alpha|= m} R_N(\partial^\alpha\mathbb{D}) + CKMR_N(\mathbb{D}) +\|f\|_{L^\infty(\Omega)}MR_N(\mathbb{D}),
 \end{equation}
 where $\partial^\alpha \mathbb{D} = \{\partial^\alpha d:~d\in \mathbb{D}\}$.
\end{theorem}
Theorem \ref{main-rademacher-bound} implies that to bound the Rademacher complexity of the set of interest, we only need to bound the Rademacher complexity of the derivatives of the dictionary $\mathbb{D}$, which is a much simpler task. In the Section \ref{rademacher-neural-networks} we will detail how to do this for the specific dictionaries corresponding to shallow neural networks.
\begin{proof}
 The proof is a straightforward application of Lemma \ref{rademacher-lemma}. We begin by noting that
 \begin{equation}
  \mathcal{L}_M \subset \sum_{|\alpha|= m} a_\alpha\cdot [\phi\circ B_M(\partial^\alpha \mathbb{D})] + a_0\cdot [\phi\circ B_M(\mathbb{D})] + f\cdot B_M(\mathbb{D}),
 \end{equation}
 where $\phi(x) = \frac{1}{2}x^2$ and $B_M(\mathbb{D}) = MB_1(\mathbb{D}) = \{f:~\|f\|_{\mathcal{K}_1(\mathbb{D})} \leq M\}$.
 
 Utilizing the first part of Lemma \ref{rademacher-lemma}, we see that for all $\alpha$
 \begin{equation}
  R_N(B_M(\partial^\alpha \mathbb{D})) \leq MR_N(\partial^\alpha \mathbb{D}).
 \end{equation}
 The third part of the Lemma, combined with the bound $\|d\|_{W^{m,\infty}} \leq C$ and the fact that $\phi$ is locally Lipschitz, imply that
 \begin{equation}
  R_N(\phi\circ B_M(\partial^\alpha \mathbb{D})) \leq CMR_N(\partial^\alpha \mathbb{D}).
 \end{equation}
 Finally, the second and fourth parts of the Lemma, combined with the bounds on $a_\alpha$ and $a_0$ complete the proof.
\end{proof}

We are primarily interested in the following corollary of this result, which uniformly bound the Monte Carlo discretization error when discretizing elliptic PDEs. The next corollary provides a bound on the discretization error incurred in such a discretization.

\begin{corollary}\label{main-rademacher-corollary}
 Suppose the empirical and true risk are defined as in \eqref{PDE-risk} and \eqref{empirical-risk} for the loss function \eqref{elliptic-loss} corresponding to the variational form of an elliptic PDE. Then, under the assumptions of Theorem \ref{main-rademacher-bound}, we have that
 \begin{equation}\begin{split}
  \mathbb{E}_{x_1,...,x_N} \sup_{v\in B_M(\mathbb{D})}|\mathcal{R}_N(v) - \mathcal{R}(v)| \leq 2CKM\sum_{|\alpha|= m} R_N(\partial^\alpha\mathbb{D})& + 2CKMR_N(\mathbb{D})+ 2\|f\|_{L^\infty(\Omega)}MR_N(\mathbb{D}).
  \end{split}
 \end{equation}
\end{corollary}
\begin{proof}
 This follows immediately by combining Theorem \ref{main-rademacher-bound} with Theorem \ref{rademacher-theorem}.
\end{proof}

\subsection{Rademacher Bounds for Neural Networks}\label{rademacher-neural-networks}
In this section, we show how the Rademacher complexity can be bounded for dictionaries corresponding to shallow neural networks. Specifically, we consider dictionaries of the form
\begin{equation}\label{neural-network-dictionary}
 \mathbb{D}_\sigma = \{\sigma(\omega\cdot x + b):~(\omega, b)\in \Theta\}\subset H^m(\Omega),
\end{equation}
where the parameter set $\Theta\subset R^{d+1}$ is compact. Of particular importance are the dictionaries corresponding to ReLU$^k$ activation functions,
$\mathbb{P}_k^d$, which were introduced in \cite{siegel2021optimal} and described in more detail in Section \ref{sec:barron-reg}. Our main result is the following bound on the Rademacher complexity. This generalizes the results of \cite{ma2022barron}, which calculate the Rademacher complexity of the unit ball in the Barron space for ReLU neural networks (see also \cite{gao2016dropout}, Theorem 2 and \cite{kakade2008complexity}, Theorem 3).
\begin{theorem}\label{neural-network-dictionary-theorem}
 Suppose that $\sigma\in W^{m+1,\infty}$. Then for any $\alpha$ with $|\alpha|\leq m$, we have
 \begin{equation}\label{complexity-bound}
  R_N(\partial^\alpha\mathbb{D}) \lesssim N^{-\frac{1}{2}},~R_{\partial,N}(\partial^\alpha\mathbb{D}) \lesssim N^{-\frac{1}{2}}
 \end{equation}
 where the implied constant is independent of $N$.
\end{theorem}
\begin{proof}
 This results follows immediately upon noting that
 \begin{equation}
  \partial^\alpha\mathbb{D} = \{\omega^{\alpha}\sigma^{(\alpha)}(\omega\cdot x + b):~(\omega, b)\in \Theta\}.
 \end{equation}
 Since $\Theta$ is a compact set, $|\omega^{\alpha}|$ is bounded. Further, since $\sigma\in W^{m+1,\infty}$, we have that $\sigma^{(\alpha)}$ is Lipschitz. Using the third point in Lemma \ref{rademacher-lemma}, we obtained
 \begin{equation}
  R_N(\partial^\alpha\mathbb{D}) \lesssim  R_N(\{\omega\cdot x + b:(\omega,b)\in \Theta\},
 \end{equation}
 and likewise for $R_{\partial,N}(\partial^\alpha\mathbb{D})$.
 
 It is well-known that the Rademacher complexity of the set of linear functions is bounded by \cite[Section 26.2]{shalev2014understanding}
 \begin{equation}
  R_N(\{\omega\cdot x + b:(\omega,b)\in \Theta\} \lesssim N^{-\frac{1}{2}},
 \end{equation}
 for any distribution on $x$ which is bounded almost surely. This applies both to the uniform distribution on $\Omega$ as well as to the uniform distribution on $\partial\Omega$, which completes the proof.
\end{proof}

\subsection{Numerical quadrature}\label{sec:num_qua}

In this section we bound the discretization error when the Gauss-Legendre quadrature rule is used to compute the energy inner-product (\ref{pga-oga-rga-argmin}) and the error $\|u-u_{n}\|_{a}$. Let $\mathcal{T}_{h} \subset \Omega$ be a partition on $\Omega$ with mesh size $h$, where $h = \mathcal{O}(N^{-\frac{1}{d}})$ and $N$ is the number of quadrature points. For each $T_{l} \in \mathcal{T}_{h}$, $l = 1,\cdots, L$, the quadrature rule satisfies
\begin{equation}
    \int_{T_{l}} p(x) dx = \sum_{i=0}^{t} p(x_{l,i}) \omega_{i}, \quad \forall p \in \mathcal{P}_{2t+1}(T_{l}),
\end{equation}
where $\mathcal{P}_{2t+1}(T)$ is the space of polynomials with degree less equal than $2t+1$. Therefore, we have
\begin{equation}\label{grid_gauss_quad}
    \int_{\Omega} f(x) dx = \sum_{l=1}^{L}\int_{T_{l}} f(x) dx = \sum_{l=1}^{L}\sum_{i=0}^{t} f(x_{l,i}) \omega_{i} = \sum_{j = 1}^{N} f(x_j) \omega_{j}, \quad \forall f \in \mathcal{P}_{2t+1}(\mathcal{T}_{h}),
\end{equation}
where $\mathcal{P}_{2t+1}(\mathcal{T}_{h}) = \lbrace g \in L^2(\Omega) : g|_{T} \in \mathcal{P}_{2t+1}(T), \forall T \in \mathcal{T}_{h} \rbrace$ is the space of piece-wise polynomial functions on the partition $\mathcal{T}_h$. We define the error operator
\begin{equation} 
    E_{t}(f) = \int_{\Omega} f(x) dx - \sum_{j = 1}^{N} f(x_j) \omega_{j}  = \sum_{l=1}^{L} E_{t,l}(f) = \sum_{l=1}^{L} \left( \int_{T_{l}} f(x) dx - \sum_{i = 0}^{t} f(x_{l,i})\omega_{i}\right)
\end{equation}
for $f\in W^{k+1,\infty}(\Omega)$.  It is clear that $E_{t,l} \in (W^{k+1,\infty}(T_{l}))^{*}$ if $k \leq 2t+1$.

\begin{theorem}\label{quad_thm}
Let $B_M(\mathbb{D}) = \lbrace f\in \mathcal{K}_1(\mathbb{D}):\|f\|_{\mathcal{K}_1(\mathbb{D})}\leq M \rbrace$. Suppose the integrand $f \in B_M(\mathbb{D})$ where $\sup_{d} \|d\|_{W^{k+1, \infty}(\Omega)} \leq C$, and the Gauss-Legendre quadrature rule is accurate for $\mathcal{P}_{k}$. Then it holds that 
\begin{equation}\label{acc_gauss_quad}
    | E_{t}(f) | \leq C_{k} C M N^{-\frac{k+1}{d}},
\end{equation}
where $t\geq [\frac{k-1}{2}]+1$ and $N$ is the number of quadrature points.
\begin{proof}
Using the Bramble-Hilbert Lemma, we get
\begin{equation}
    | E_{n,\hat{T}}(\hat{f}) | \leq C \| E_{n,\hat{T}} \|_{W^{k+1,\infty}(\hat{T})}^{*} | \hat{f} |_{W^{k+1,\infty}(\hat{T})}
\end{equation}
on the reference domain $\hat{T}$. By the standard scaling argument, it gives on $\mathcal{T}_{h}$ that
\begin{equation}
    | E_{n}(f) | \leq C_{k} h^{k+1} \|f\|_{W^{k+1,\infty}(\Omega)} \leq C_{k} C M h^{k+1}.
\end{equation}
The relation $h=\mathcal{O}(N^{-\frac{1}{d}})$ gives the result.
\end{proof}
\end{theorem}

The accuracy with respect to $N$ in (\ref{acc_gauss_quad}) allows us to use the numerical quadrature (\ref{grid_gauss_quad}) to compute the generalization errors such as $\|u-u_{n}\|_{L^2}$ and $\|u-u_{n}\|_{a}$. Since $\text{ReLU}^{k}(\omega \cdot x + b)$ is in the  $W^{k,\infty}(\Omega)$ Sobolev space where $\Omega$ is bounded, then we have $\mathcal{K}_1(\mathbb{P}_k^d) \subset W^{k,\infty}(\Omega)$ which satisfies the condition of Theorem \ref{quad_thm}. 

\section{Balancing the Error Terms}
In this section, we combine the estimates of the optimization, discretization, and modelling errors obtained in the previous sections to obtain a complete convergence theory and explain how to choose the hyperparameters $N$ and $n$ in each of the different situations discussed. Specifically, when using the relaxed greedy algorithm we have the following convergence theorem.
\begin{theorem}
    Suppose that the Relaxed Greedy Algorithm (RGA) \eqref{relaxed-greedy} is applied to the discretized loss function $\mathcal{R}_N$ corresponding to the risk formulation \eqref{PDE-risk} of the PDE \eqref{2mPDE}. Suppose further that the solution $u$ satisfies $\|u\|_{\mathcal{K}_1(\mathbb{D})} \leq M$ and that
    \begin{itemize}
        \item Monte Carlo quadrature is used and the assumptions of Theorem \ref{main-rademacher-bound} are satisfied. If the dictionary $\mathbb{D}$ satisfies $R_N(\partial^\alpha \mathbb{D}) \lesssim N^{-\frac{1}{2}}$ and we set $N = n^2$, we have the convergence rate
        \begin{equation}
            \mathcal{R}(u_{\Theta,N}) - \mathcal{R}(u) \lesssim n^{-1}.
        \end{equation}
        In particular, since the objective error is comparable to the squared $H^m$-error, we also have
        \begin{equation}
            \|u_{\Theta,N} - u\|_{H^m(\Omega)} \lesssim n^{-\frac{1}{2}}.
        \end{equation}
        \item Numerical quadrature of order $k$ is used and $\mathbb{D}$ is uniformly bounded in $W^{k+1,\infty}(\Omega)$. If we set $N = n^\frac{d}{2(k+1)}$, then we have
        \begin{equation}
            \mathcal{R}(u_{\Theta,N}) - \mathcal{R}(u) \lesssim n^{-1}.
        \end{equation}
        We also have
        \begin{equation}
            \|u_{\Theta,N} - u\|_{H^m(\Omega)} \lesssim n^{-\frac{1}{2}}.
        \end{equation}
    \end{itemize}
\end{theorem}
Note in particular that the assumptions of this theorem hold when using shallow neural network dictionaries.
We remark that when using the orthogonal greedy algorithm, the $\mathcal{K}_1(\mathbb{D})$-norm cannot be a priori bounded and this is a missing ingredient in obtaining an a priori bound on the discretization error. Nonetheless, we obtain good performance in practice for the orthogonal greedy algorithm. In addition, if the solution $u$ does not satisfy the bound $\|u\|_{\mathcal{K}_1(\mathbb{D})} \leq M$, then the method will nonetheless still optimize the risk over this set. In this case, a bound on the error can be obtained by determining how efficiently the solution can be approximated by a function $u$ which satisfies $\|u\|_{\mathcal{K}_1(\mathbb{D})} \leq M$. The proper theory here is the theory of interpolation spaces (see for instance \cite{devore1993constructive}, Chapter 6, or \cite{barron2008approximation}), but we do not go into detail here.

\section{Numerical experiments}\label{sec:numerical}
In this section, we provide numerical experiments demonstrating the effectiveness of the proposed algorithms on a variety of problems. For all the experiments below, the energy functions are discretized using Gaussian quadrature with the default setting $t=2$ and $L=4000$ in \eqref{grid_gauss_quad} for 1D and $t=2\times 2, L = 400\times 400$ for 2D. 
For simplicity, we use $u_n$ to denote the numerical solution and use $u$ to represent the analytical solution in this section. We also define $\Vert u-u_n \Vert_{a}$ to be the discretization error in the energy norm which is defined in Section \ref{OGA}. For a specific type of second order elliptic equation discussed in (\ref{energy-function}), the energy norm is identical to the $H^1$ norm. In addition, the discretization error in the $L^2$ norm is also reported. For the detailed computation of these norms, we refer to the technique presented in Section \ref{sec:num_qua} and the definition of norms \eqref{discrete-energy-norm}. We remark that when calculating these norm we used a new and large set of quadrature points different from the ones used for training the network.

Section \ref{sec:numerical} is organized as follows. In Example \ref{ex:1} we test our method on a simple one-dimensional problem with both Dirichlet and Neumann boundary conditions. We do this with both the energy and PINN loss formulation of the problem and compare our method with the common SGD, ADAM and L-BFGS optimizers to demonstrate its effectiveness. 
In Example \ref{ex:2} we present a 1D benchmark to verify the empirical adaptive property of greedy algorithms. Next we consider solving high order and high dimensional PDEs using the OGA. Examples \ref{ex:2d-elliptic-oga} and \ref{ex:7} confirm our theoretical convergence rates for two dimensional elliptic problems with second and fourth order. In Example \ref{ex:9} we develop a method using a restricted dictionary designed for high dimensional problems. We show that our method can tackle high-dimensional problems as long as the solution is well-approximated by the convex hull of the fictionary. Finally, we give an examples of non-linear PDEs in Section \ref{NPDE} using the relaxed greedy algorithm (RGA). We note that Theorem \ref{RGA_convergence} holds for any convex and smooth energy function. Therefore, we can get convergence for non-linear equations provided the equation has a variation formulation with a convex energy.

\subsection{Linear PDEs}

\begin{example}[1D elliptic equation]\label{ex:1}
We consider the 1D elliptic equation
\begin{equation}\label{1-d-elliptic-eq}
    \begin{aligned}
       - u^{\prime\prime} + u &= f, ~x \in (-1,1),\\
       u^{\prime}(-1) &= u^{\prime}(1) = 0,
    \end{aligned}
\end{equation}
 with the source term $f=\big(1+\pi^2\big)\cos \big(\pi x\big)$ which has the analytical solution $u(x) = \cos \big(\pi x\big)$. The energy function is discretized using Guassian quadrature with $t=2$ and $L=4000$ in \eqref{grid_gauss_quad} and the discrete energy is minimized using the orthogonal greedy algorithm OGA with dictionary $\mathbb{P}_{2}^{1}$ (i.e. corresponding to $\operatorname{ReLU}^2$).
The convergence rate is shown in Table \ref{1dex1}. We obtain second order convergence in $H^1((-1,1))$ which matches the theoretical convergence rate of the orthogonal greedy algorithm. In addition, we obtain third order convergence in $L^2((-1,1))$ which matches the theoretically predicted approximation rates of shallow neural networks \cite{siegel2021sharp}.

\begin{table}[H]
\centering
\begin{tabular}{|c|c|c|c|c|}
\hline
$n$ & $\Vert u-u_n \Vert_{L^2}$ & order$(n^{-3})$ & $\Vert u-u_n \Vert_{H^1}$ & order$(n^{-2})$ \\ \hline
   16 & 7.86e-04 &   -     & 2.79e-02 &     -       \\ \hline
   32 & 7.70e-05 &   3.35  & 5.89e-03 &     2.24    \\ \hline
   64 & 8.45e-06 &   3.19  & 1.36e-03 &     2.11    \\ \hline
  128 & 9.68e-07 &   3.13  & 3.22e-04 &     2.08    \\ \hline
  256 & 1.18e-07 &   3.04  & 7.81e-05 &     2.04    \\ \hline
  512 & 1.44e-08 &   3.03  & 1.94e-05 &     2.01    \\ \hline
 1024 & 1.83e-09 &   2.97  & 4.86e-06 &     1.99    \\ \hline
 2048 & 2.50e-10 &   2.88  & 1.28e-06 &     1.93    \\ \hline
\end{tabular}
\caption{$L^2$ and $H^1$ numerical error of OGA  v.s. the number of neurons $n$ for {\bf Example 1}.}\label{1dex1}
\end{table}

Next we consider the same equation with Dirichlet boundary conditions and consider the forcing term $f(x)=\big(1+\frac{\pi^2}{4}\big)\cos(\frac{\pi}{2}x)$ so that the analytical solution is given by $u(x) = \cos(\frac{\pi}{2}x)$. We use the orthogonal greedy algorithm with $\mathbb{D}=\mathbb{P}_2^1$ to minimize a discretized version of the penalized energy $\mathcal{R}_{N,\delta}$, which is discretized using the same Gaussian quadrature. To balance the errors, we let $\delta$ scale as $n^{-2}$. The convergence order is given in Table \ref{example1_DBC} and matches the expected rate obtained by combining the convergence order of the orthogonal greedy algorithm with the error incurred by the penalization. %
\begin{table}[H]
\centering
\begin{tabular}{|c|c|c|c|c|}
\hline
$n$ & $\Vert u-u_n \Vert_{L^2}$ & order & $\Vert u-u_n \Vert_{a,\delta}$ & order($n^{-1}$)  \\ \hline
16 & 6.72e-04 & -    & 4.40e-02 & -    \\ \hline
32 & 1.70e-04 & 2.01 & 2.20e-02 & 1.00 \\ \hline 
64 & 4.17e-05 & 2.00 & 1.10e-02 & 1.00 \\ \hline 
128 & 1.04e-05 & 2.00 & 5.49e-03 & 1.00 \\ \hline 
256 & 2.63e-06 & 1.99 & 2.75e-03 & 1.00 \\ \hline
512 & 8.10e-07 & 1.70 & 1.37e-03 & 1.00
\\ \hline
\end{tabular}
\caption{Numerical results of OGA for {\bf Example 1} with Dirichlet boundary condition. Here we take $\delta = 0.1\times n^{-2}$.}\label{example1_DBC}
\end{table}

We also use the first example with Neumann's boundary conditions to compare with the deep Ritz method \cite{weinan2018deep} using SGD and ADAM \cite{kingma2014adam} as the optimizers. The numerical solution of the deep Ritz method is represented by a single hidden layer neural network with $\operatorname{ReLU}^2$ activation function. We run both SGD and ADAM optimizers for $10000$ epochs using Gauss quadrature points  with random initialization. The initial learning rate for each experiment is $1\times 10^{-3}$ and is decreased by $5$ every $3000$ epochs. The numerical errors shown in Table \ref{compare:ex4} are the average results of 30 independent experiments, where we can see that both SGD and ADAM do not achieve any convergence order numerically as $n$ gets larger (i.e. the size of the network gets larger). 

\begin{table}[H]
\centering
\begin{tabular}{|c|cccc|cccc|}
\hline
     & \multicolumn{4}{c|}{Adam}                                                                                            & \multicolumn{4}{c|}{SGD}                                                                                             \\ \hline
$n$    & \multicolumn{1}{c|}{$\|u-u_n\|_{L^2}$} & \multicolumn{1}{c|}{order} & \multicolumn{1}{c|}{$\|u-u_n\|_{H^1}$} & order & \multicolumn{1}{c|}{$\|u-u_n\|_{L^2}$} & \multicolumn{1}{c|}{order} & \multicolumn{1}{c|}{$\|u-u_n\|_{H^1}$} & order \\ \hline
16   & \multicolumn{1}{c|}{1.61e-02}      & \multicolumn{1}{c|}{-}  & \multicolumn{1}{c|}{1.45e-01}      & -     & \multicolumn{1}{c|}{1.30e-02}      & \multicolumn{1}{c|}{-}  & \multicolumn{1}{c|}{1.52e-01}      & -     \\ \hline
32   & \multicolumn{1}{c|}{3.71e-03}      & \multicolumn{1}{c|}{2.12}  & \multicolumn{1}{c|}{5.84e-02}      & 1.32  & \multicolumn{1}{c|}{9.35e-03}      & \multicolumn{1}{c|}{0.47}  & \multicolumn{1}{c|}{1.13e-01}      & 0.43  \\ \hline
64   & \multicolumn{1}{c|}{1.80e-03}      & \multicolumn{1}{c|}{1.04}  & \multicolumn{1}{c|}{3.46e-02}      & 0.76  & \multicolumn{1}{c|}{7.11e-03}      & \multicolumn{1}{c|}{0.39}  & \multicolumn{1}{c|}{8.64e-02}      & 0.38  \\ \hline
128  & \multicolumn{1}{c|}{5.52e-04}      & \multicolumn{1}{c|}{1.70}  & \multicolumn{1}{c|}{1.43e-02}      & 1.27  & \multicolumn{1}{c|}{5.91e-03}      & \multicolumn{1}{c|}{0.27}  & \multicolumn{1}{c|}{7.22e-02}      & 0.26  \\ \hline
256  & \multicolumn{1}{c|}{2.26e-04}      & \multicolumn{1}{c|}{1.29}  & \multicolumn{1}{c|}{6.99e-03}      & 1.04  & \multicolumn{1}{c|}{5.75e-03}      & \multicolumn{1}{c|}{0.04}  & \multicolumn{1}{c|}{7.03e-02}      & 0.04  \\ \hline
512  & \multicolumn{1}{c|}{1.88e-04}      & \multicolumn{1}{c|}{0.27}  & \multicolumn{1}{c|}{3.90e-03}      & 0.84  & \multicolumn{1}{c|}{4.41e-03}      & \multicolumn{1}{c|}{0.38}  & \multicolumn{1}{c|}{5.40e-02}      & 0.38  \\ \hline
1024 & \multicolumn{1}{c|}{2.09e-04}      & \multicolumn{1}{c|}{-0.16} & \multicolumn{1}{c|}{2.56e-03}      & 0.61  & \multicolumn{1}{c|}{1.52e-03}      & \multicolumn{1}{c|}{1.54}  & \multicolumn{1}{c|}{1.99e-02}      & 1.43  \\ \hline
2048 & \multicolumn{1}{c|}{4.11e-04}      & \multicolumn{1}{c|}{-0.97} & \multicolumn{1}{c|}{2.51e-03}      & 0.03  & \multicolumn{1}{c|}{3.22e-03}      & \multicolumn{1}{c|}{-1.09} & \multicolumn{1}{c|}{3.56e-02}      & -0.84 \\ \hline
\end{tabular}
\caption{The numerical convergence test of the deep Ritz method with both Adam and SGD optimizers on one-hidden-layer $\operatorname{ReLU}^2$ neural network v.s. the number of neurons $n$ for {\bf Example 1}.}\label{compare:ex4}
\end{table}

Next, we compare with the widely used PINN method \cite{raissi2019physics} which has been proven exceptionally successful in practical engineering applicatons. Specifically, we optimize a discretization of the PINN risk \eqref{PINN-risk}, given by 
\begin{equation}
MSE = MSE_f + MSE_{bc},
\end{equation}
where
\begin{equation}\label{loss-pinn}
    MSE_f = \dfrac{1}{N_f} \sum_{i=1}^{N_f} |-\Delta u_n(x_i) + u_n(x_i) - f(x_i)|^2 \hbox{~and~}
    MSE_{bc} = |u^{\prime}(-1)|^2 + |u^{\prime}(1)|^2.
\end{equation}
Here the collocation points $\lbrace x_i \rbrace_{i=1}^{N_f} $ are randomly chosen from the uniform distribution on $[-1,1]$ where $N_f = 10000$. The numerical solution $u_n$ is computed by optimizing the MSE loss with two stages. First the network is trained using ADAM's optimizer up to 10000 steps. In the next stage we change the optimizer into L-BFGS, where the learning rate is determined by the line search with the strong Wolfe's condition, and stop when the update of MSE loss is less than $10^{-16}$. Since the gradient based training method requires the computation of first order derivatives, we change the activation function into $\operatorname{ReLU}^3$ to satisfy the regularity demand. The result shown in Table~\ref{tab:pinn} is the mean error of $30$ independent experiments. Similar to the results of the Deep Ritz method, we do not observe any stable numerical convergence with respect to $n$.

\begin{table}[H]
\centering
\begin{tabular}{|c|c|c|c|c|c|c|}
\hline
n   & PINN-loss & order & $\Vert u-u_n \Vert_{L^2}$ & order & $\Vert u-u_n \Vert_{H^1}$ & order \\ \hline
16 & 9.19e-02 & -    & 5.08e-03 & -    & 3.17e-02 & -    \\ \hline 
32 & 7.65e-02 & 0.27 & 4.11e-03 & 0.31 & 2.54e-02 & 0.32 \\ \hline 
64 & 1.86e-03 & 5.37 & 1.44e-04 & 4.84 & 1.62e-03 & 3.97 \\ \hline 
128 & 1.84e-04 & 3.33 & 5.20e-05 & 1.47 & 3.02e-04 & 2.43 \\ \hline 
256 & 1.13e-05 & 4.03 & 5.22e-06 & 3.32 & 4.50e-05 & 2.74 \\ \hline 
512 & 5.58e-06 & 1.02 & 3.58e-06 & 0.54 & 3.10e-05 & 0.54 \\ \hline 
1024 & 3.28e-05 & -2.55 & 1.73e-05 & -2.27 & 1.28e-04 & -2.04 \\ \hline 
2048 & 1.52e-05 & 1.11 & 1.30e-05 & 0.42 & 8.15e-05 & 0.65 \\ \hline 
\end{tabular}
\caption{The numerical convergence test of the PINN model with L-BFGS optimizer on one-hidden-layer $\operatorname{ReLU}^3$ neural network v.s. the number of neurons $n$ for {\bf Example 1}.}\label{tab:pinn}
\end{table}

For comparison, next we use the orthogonal greedy algorithm to train the neural network using the same loss function \eqref{loss-pinn} with $N_f = 10000$ and the dictionary $\mathbb{P}_3^1$. We compute the numerical errors using the quadrature with a number of points that is large enough to get a good accuracy. We observe convergence for both the loss function and the numerical error in Table \ref{tab:oga-pinn}:
\begin{table}[H]
\centering
\begin{tabular}{|c|c|c|c|c|c|c|}
\hline
n   & PINN-loss & order & $\Vert u-u_n \Vert_{L^2}$ & order & $\Vert u-u_n \Vert_{H^1}$ & order \\ \hline
16 & 2.64e-03 & -    & 5.05e-04 & -    & 2.02e-03 & 5.28 \\ \hline 
32 & 1.50e-04 & 4.14 & 1.04e-04 & 2.29 & 2.54e-04 & 2.99 \\ \hline 
64 & 8.10e-06 & 4.21 & 2.28e-05 & 2.18 & 4.43e-05 & 2.52 \\ \hline 
128 & 5.10e-07 & 3.99 & 5.03e-06 & 2.18 & 1.26e-05 & 1.81 \\ \hline 
256 & 3.16e-08 & 4.01 & 3.49e-06 & 0.53 & 5.40e-06 & 1.22 \\ \hline 
512 & 1.98e-09 & 4.00 & 5.45e-07 & 2.68 & 1.76e-06 & 1.62 \\ \hline 
1024 & 1.11e-10 & 4.16 & 1.81e-07 & 1.59 & 5.74e-07 & 1.62 \\ \hline 
2048 & 5.54e-12 & 4.32 & 6.67e-08 & 1.44 & 2.14e-07 & 1.43 \\ \hline 
\end{tabular}
\caption{The loss function and numerical error of OGA in $L^2$ and $H^1$ norms  v.s. the number of neurons $n$ for {\bf Example 1}.}\label{tab:oga-pinn}
\end{table}

\end{example}

\begin{example}[Adaptivity in 1D]\label{ex:2}
Next, we test the OGA using the dictionary $\mathbb{P}_{2}^{1}$ on the 1D elliptic equation \eqref{1-d-elliptic-eq} where $f$ is chosen so that the exact solution is given by:
\begin{align}
u(x) = (1+x)^2(1-x^2)\left( 0.5 \exp \left(-\frac{(x+0.5)^2}{K}\right) + \exp \left(-\frac{x^2}{K}\right) + 0.5 \exp \left(-\frac{(x-0.5)^2}{K}\right)\right),
\end{align}
for $x\in \Omega = (-1,1)$ and $K=0.01$. The exact solution has three peaks as shown in Fig. \ref{fig:ex5}. In this example, we illustrate the adaptivity of the neural network discretization by identifying the grid points $x=(x_1,\cdots,x_N)^T$ such that $w_1x+b_1=0$, i.e. where the second derivative of the numerical solution changes. %
Since $u(x)$ has three peaks, we see that the grid points are gathered mainly at places with a larger curvature and are adaptive to fit the three peaks shown in Fig. \ref{fig:ex5}. Furthermore, both the numerical error and the convergence order are shown in Table \ref{tab:ex5}, where we see the theoretical convergence order achieved numerically. Note that the adaptivity is mainly a result of the neural network function class we are using and is likely to be present for other training algorithms as well. This merely demonstrates that a greedy algorithm is able to adapt to sharp changes in the solution.

\begin{figure}[H]
\centering
\includegraphics[width=0.65\textwidth]
{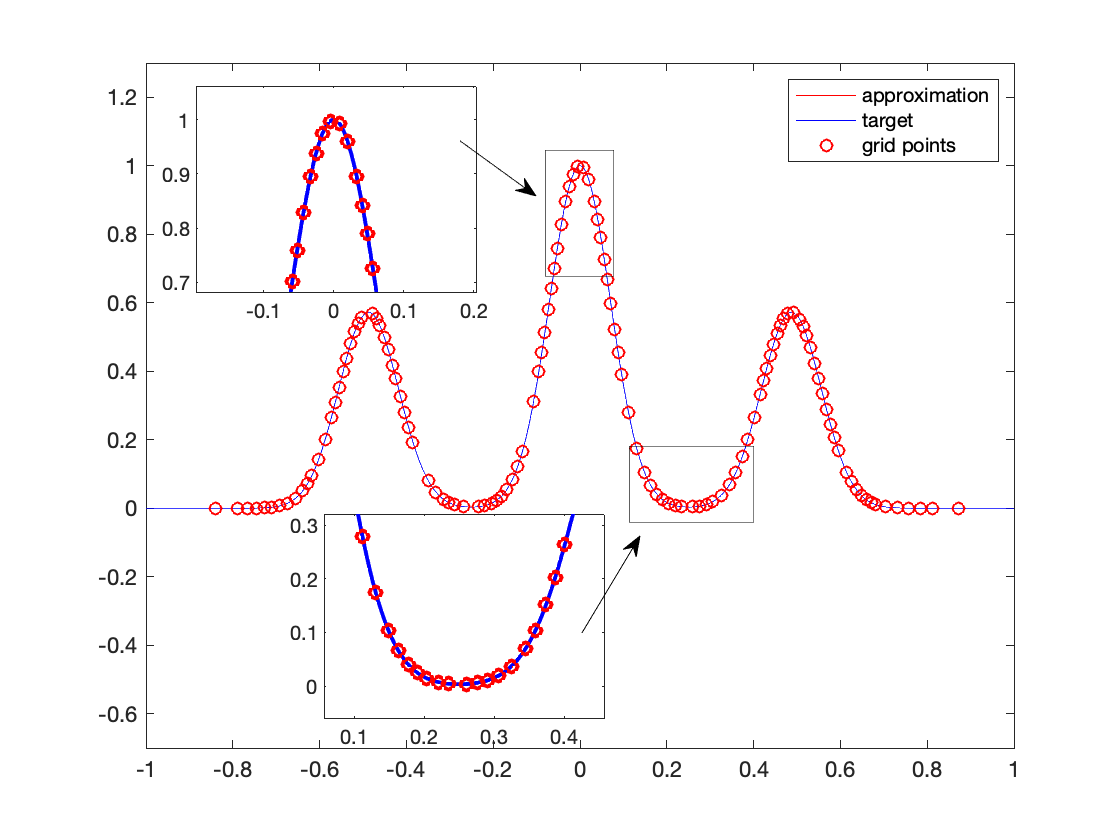}
\caption{Grid points of a 1-hidden layer neural network solution with $N=128$ for {\bf Example 2}.}\label{fig:ex5}
\end{figure}

\begin{table}[H]
\centering
\begin{tabular}{|c|c|c|c|c|}
\hline
$n$ & $\Vert u-u_n \Vert_{L^2}$ & order$(n^{-3})$ & $\Vert u-u_n \Vert_{H^1}$ & order$(n^{-2})$ \\ \hline
16  & 5.05e-02  & -     & 1.43e+00  & -  
      \\ \hline
32  & 1.96e-03  & 4.69  & 1.62e-01  & 3.14  \\ \hline
64  & 2.08e-04  & 3.23  & 3.93e-02  & 2.05  \\ \hline
128 & 1.99e-05  & 3.39  & 8.42e-03  & 2.22  \\ \hline
256 & 2.34e-06  & 3.09  & 2.04e-03  & 2.04  \\ \hline
512 & 2.85e-07  & 3.04  & 4.83e-04  & 2.08  \\ \hline
\end{tabular}
\caption{$L^2$ and $H^1$ numerical errors and convergence orders of OGA for {\bf Example 2}.}\label{tab:ex5}
\end{table}
\end{example}

\begin{example}[2D elliptic equation]\label{ex:2d-elliptic-oga}
We consider the 2D elliptic equation in $\Omega = (0,1)^2$ given by
\begin{equation}
    \begin{aligned}
       -\Delta u + u &= f, ~x \in (0,1)^2\\
       \dfrac{\partial u}{\partial n} &= 0, ~ x \in \partial (0,1)^2.
    \end{aligned}
\end{equation}
where the right hand side $f$ is chosen so that the exact solution is given by $u(x,y)=\cos(2\pi x)\cos(2\pi y)$. We discretize the energy using Gaussian quadrature of order $2$ with $400$ points in each direction. We optimize the discrete energy using the orthogonal greedy algorithm with the dictionary $\mathbb{P}_{2}^{2}$. 
The convergence orders with both $L^2$ and $H^1$ errors are shown in Table \ref{tab:ex6} and confirm the theoretical orders of $1.75$ and $1.25$ for $L^2$ and $H^1$ errors, respectively. Note that the convergence order appears even to be slightly better than predicted by our theory. This demonstrates that we have only proved an upper bound, and for certain dictionaries the convergence rate of the orthogonal greedy algorithm may even by faster than predicted by Theorem \ref{OGA-convergence-theorem}. Due to the computational difficulty of solving the argmax subproblem \eqref{pga-oga-argmax} to the required high degree of accuracy, we were not able to run this example beyond $356$ neurons with the variational loss.

\begin{table}[H]
\centering
\begin{tabular}{|c|c|c|c|c|}
\hline
$n$ & $\Vert u-u_n \Vert_{L^2}$ & order$(n^{-1.75})$ & $\Vert u-u_n \Vert_{H^1}$ & order$(n^{-1.25})$ \\ \hline
16  & 5.13e-02  & -     & 9.74e-01  & -  
      \\ \hline
32  & 9.72e-03  & 2.40  & 3.07e-01  & 1.66  \\ \hline
64  & 2.26e-03  & 2.10  & 1.07e-01  & 1.53  \\ \hline
128 & 5.86e-04  & 1.95  & 4.04e-02  & 1.40  \\ \hline
256 & 1.42e-04  & 2.04  & 1.51e-02  & 1.42  \\ \hline
356 & 7.68e-05  & 1.87  & 9.82e-03  &
1.30  \\ \hline
\end{tabular}
\caption{Convergence order test of OGA with both $L^2$ and $H^1$ errors for {\bf Example 3}.}\label{tab:ex6}
\end{table}

Next we consider the dictionary $\mathbb{P}_3^2$ and optimize the PINN formulation instead of the energy formulation of the problem. Specifically, we optimize a discretization of the PINN risk \eqref{PINN-risk}, given by
$$MSE = MSE_f + MSE_{bc},$$
where $MSE_f$ is the discrete $L^2$-residual in the domain $(0,1)^2$:
\begin{equation}
    MSE_f = \dfrac{1}{N_f} \sum_{i=1}^{N_f} \vert -\Delta u_n(x_i^f) + u_n(x_i^f) - f(x_i^f) \vert^2,
\end{equation}
and $MSE_{bc}$ is the residual on the boundary $\partial (0,1)^2$:
\begin{equation}
    MSE_{bc} = \dfrac{1}{N_{bc}} \sum_{j=1}^{N_{bc}} \left\vert \dfrac{\partial u_n}{\partial n}(x_j^{bc}) \right\vert^2.
\end{equation}
Here we take $N_f=20000$ and $2000$ samples on each edge of $\partial (0,1)^2$ so that $N_{bc}=8000$. The samples $\lbrace x_i^f \rbrace_{i=1}^{N_f}$ and $\lbrace x_j^{bc} \rbrace_{j=1}^{N_{bc}}$ are randomly chosen in the corresponding domains from the uniform distribution. The following table shows the numerical result where the neural network is trained by OGA. We compute the numerical errors using the quadrature with a number of points large enough to get a good accuracy. We see that although the PINN loss converges with a good order as expected, the solution errors converge somewhat more slowly and less reliably than the loss. However, a good accuracy is nonetheless finally obtained even in terms of the solution error (see Table \ref{tab:2m2d-pinn-oga}).
\begin{table}[H]
\centering
\begin{tabular}{|c|c|c|c|c|c|c|}
\hline
n   & PINN-loss & order & $\Vert u-u_n \Vert_{L^2}$ & order & $\Vert u-u_n \Vert_{H^1}$ & order \\ \hline
16 & 9.51e+01 & -    & 2.93e-01 & -    & 1.41e+00 & -    \\ \hline 
32 & 1.34e+01 & 2.83 & 7.04e-02 & 2.06 & 4.09e-01 & 1.79 \\ \hline 
64 & 1.79e+00 & 2.90 & 1.49e-02 & 2.24 & 1.09e-01 & 1.91 \\ \hline 
128 & 1.91e-01 & 3.23 & 4.67e-03 & 1.68 & 3.78e-02 & 1.52 \\ \hline 
256 & 2.67e-02 & 2.84 & 5.88e-04 & 2.99 & 8.13e-03 & 2.22 \\ \hline 
512 & 3.27e-03 & 3.03 & 5.60e-04 & 0.07 & 2.17e-03 & 1.90 \\ \hline 
1024 & 5.33e-04 & 2.62 & 5.22e-04 & 0.10 & 7.24e-04 & 1.58 \\ \hline 
2048 & 7.96e-05 & 2.74 & 9.43e-05 & 2.47 & 2.08e-04 & 1.80 \\ \hline 
\end{tabular}
\caption{The loss function and numerical error of OGA in $L^2$ and $H^1$ norms  v.s. the number of neurons $n$ for {\bf Example 3} with the PINN loss.}\label{tab:2m2d-pinn-oga}
\end{table}
\end{example}

\begin{example}[2D fourth-order differential equation]\label{ex:7}
We consider the fourth-order equation 
\begin{equation}
    \begin{aligned}
       \Delta^2 u + u &= f, ~x \in (-1,1)^2,\\
       B^0_N(u) &= 0,~ x \in \partial (-1,1)^2,\\
       B^1_N(u) &= 0,~ x \in \partial (-1,1)^2.
    \end{aligned}
\end{equation}
We choose the right hand side so that the exact solution is $u(x,y)=(x^2-1)^4(y^2-1)^4$. We discretize the energy using Gaussian quadrature of order $2$ with $400$ points in each direction and using the orthogonal greedy algorithm with the dictionary $\mathbb{P}_{3}^{2}$ to optimize the discrete energy. We plot
the convergence orders for the $L^2$ energy norms in Table \ref{tab:ex11}. Each of these errors is calculated by using finer Gaussian quadrature. For this example, we were again only able to run the algorithm with $256$ neurons due to the computational difficulty of the argmax subproblem \eqref{pga-oga-argmax}.
\begin{table}[H]
\centering
\begin{tabular}{|c|c|c|c|c|}
\hline
$n$ & $\Vert u-u_n \Vert_{L^2}$ & order$(n^{-2.25})$ & $\Vert u-u_n \Vert_{a}$ & order$(n^{-1.25})$ \\ \hline
16 & 1.72e-01 & -    & 4.84e+00 & -    \\ \hline 
32 & 1.89e-02 & 3.18 & 1.60e+00 & 1.59 \\ \hline 
64 & 3.36e-03 & 2.50 & 5.85e-01 & 1.46 \\ \hline 
128 & 4.24e-04 & 2.99 & 2.04e-01 & 1.52 \\ \hline 
256 & 8.25e-05 & 2.36 & 8.19e-02 & 1.32 \\ \hline 
\end{tabular}
\caption{The convergence order of OGA with $\Vert \cdot \Vert_{L^2}$ and $\Vert \cdot \Vert_a$ errors  for  {\bf Example 4}.}\label{tab:ex11}
\end{table}

\end{example}

\begin{example}[High-dimensional example]\label{ex:9}
We consider the following 10d elliptic equation:
\begin{equation}
    \begin{aligned}
       -\nabla\cdot (\alpha \nabla u) + u &= f, ~x \in (0,1)^{10}\\
       \dfrac{\partial u}{\partial n} &= 0, ~ x \in \partial (0,1)^{10},
    \end{aligned}
\end{equation}
with
\begin{equation}
    \alpha = \sqrt{1+ \sum_{i=1}^{10}(x_i-\dfrac{1}{2})^2}.
\end{equation}
We choose the right hand side $f$ so that the exact solution is given by
\begin{equation}
    u = \sum_{i=1}^{10}\cos(\pi x_i).
\end{equation}
In order to be able to solve the argmax problem arising in \eqref{pga-oga-argmax} we use the restricted dictionary
\begin{equation}
    \mathbb{P}_2^{10,r} = \{\sigma(\omega\cdot x + b),~\omega = \pm e_i,~i=1,..,10,~b\in [-2,2]\},
\end{equation}
where $\sigma = \text{ReLU}^2$. We note that the solution $u$ was specifically chosen to lie in the convex hull of $\mathbb{P}_2^{10,r}$, which is given by
\begin{equation}
    B_1(\mathbb{P}_2^{10,r}) = \left\{f(x) = \sum_{i=1}^{10}f_i(x_i),~\sum_{i=1}^{10} \|f_i^{2}\|_{BV} \leq 1\right\}.
\end{equation}
Note that the equation itself is not separable due to the complicated coefficients $\alpha$, and all that is required for the method to work is that the solution be well approximated by the dictionary.
We discretize the energy using 100 million quasi-Monte-Carlo samples in $[0,1]^{10}$ generated by the Halton sequence, and optimize the energy using the orthogonal greedy algorithm with dictionary $\mathbb{P}_2^{10,r}$. The results are shown in table \ref{tab:ex9}. The point of this example is to demonstrate that the proposed method converges as expected even in high-dimensions as long as the solution is well-approximated by the dictionary $\mathbb{D}$. For this example we were only able to run the algorithm for $256$ iterations due to the very large number of quasi-Monte Carlo samples required for the high dimensional problem.
\begin{table}[H]
\centering
\begin{tabular}{|c|c|c|c|c|}
\hline
$n$ & $\Vert u-u_n \Vert_{L^2}$ & order$(n^{-3})$ & $\Vert u-u_n \Vert_{H^1}$ & order$(n^{-2})$ \\ \hline
16 & 5.02e-01 & -    & 3.18e+00 & -    \\ \hline
32 & 4.70e-02 & 3.42 & 5.99e-01 & 2.41 \\ \hline 
64 & 4.63e-03 & 3.34 & 1.10e-01 & 2.44 \\ \hline 
128 & 4.44e-04 & 3.38 & 2.27e-02 & 2.28 \\ \hline 
256 & 5.41e-05 & 3.04 & 5.19e-03 & 2.13 \\ \hline
\end{tabular}
\caption{The convergence order of OGA on a high-dimensional problem with $\Vert \cdot \Vert_{L^2}$ and $\Vert \cdot \Vert_{H^1}$ errors for {\bf Example 5}.}
\label{tab:ex9}
\end{table}
\end{example}

\subsection{Nonlinear PDEs}\label{NPDE}
Next, we test the convergence order of the RGA on a nonlinear Poisson-Boltzmann PDE to confirm the theoretically derived first order convergence in Theorem \ref{RGA_convergence}. We also test both sigmoid and ReLU$^2$ activation functions and compare RGA with OGA in the 1D example. For all the RGA's results, we report the generalization error defined in the left hand side of \eqref{generalization-error-decompose} in a relative sense, i.e., $\frac{\mathcal{R}(u_n) - \mathcal{R}(u)}{\mathcal{R}(u_0) - \mathcal{R}(u)}$, which is computed by using the numerical quadrature scheme.

\begin{example}[2D Poisson-Boltzmann equation, \cite{2dPB2007}] \label{ex:poisson_bolztmann}
We consider the 2D  Poisson-Boltzmann equation
on the sphere $\lbrace (x,y)| x^2+y^2\leq 4 \rbrace$, with Neumann boundary conditions, namely,
 \begin{equation}\label{eq:allen_cahn}
 \left\{
 	\begin{aligned}
   &  -\Delta u  +  \kappa \sinh(u) = f,   \quad (x,y) \in \Omega = B_2(0), \\
   &   \dfrac{\partial u}{\partial n} =0,  \quad (x,y) \in \partial B_2(0).
    \end{aligned}
 	\right.
\end{equation}
The  energy functional for this problem is
$$ \mathcal{R}(u) = \int_{\Omega} \left(\dfrac{1}{2} \vert \nabla u \vert^2 + \kappa \cosh(u) - f u \right) dx,$$
which is a strictly convex and coercive energy with respect to $u$ as long as $\kappa > 0$ and this implies the existence and uniqueness of the solution. We set $\kappa = 1$ and consider the radially symmetric solution 
$u(x,y) = \cos(\frac{\pi}{2}\sqrt{x^2+y^2})$, which gives the source terms $f$. We use Monte-Carlo quadrature with the number of samples $N=O(n^2)=\frac{n^2}{10}$ to approximate the integration. The dictionary for the RGA algorithm is taken as
$$\mathbb{D} = \lbrace \sigma(w_1x+w_2y + b) | (w_1,w_2,b) \in [-20,20]^3 \rbrace,$$
 where $\sigma$ is the sigmoidal activation function and we set $M=20$ in \eqref{relaxed-greedy}. The convergence order test is shown in Table \ref{tab:ex12} and the numerical solution is plotted in Fig. \ref{fig:ex12}.

\begin{table}[H]
\centering
\begin{tabular}{|c|c|c|}
\hline
$n$     & $\frac{\mathcal{R}(u_n) - \mathcal{R}(u)}{\mathcal{R}(u_0) - \mathcal{R}(u)}$ & order$(n^{-1})$ \\ \hline
16   & 8.18e+00    & -     \\ \hline
32   & 4.19e+00    & 0.96  \\ \hline
64   & 2.96e+00    & 0.50  \\ \hline
128  & 6.95e-01    & 2.09  \\ \hline
256  & 2.54e-01    & 1.45  \\ \hline
512  & 7.70e-02    & 1.72  \\ \hline
1024 & 2.90e-02    & 1.41  \\ \hline
2048 & 1.39e-02    & 1.06  \\ \hline
\end{tabular}
\caption{Convergence order of RGA for the 2D Poisson-Boltzmann equation in {\bf Example 6}.}\label{tab:ex12}
\end{table}

\begin{figure}[H]
\centering
\includegraphics[width=0.65\textwidth]
{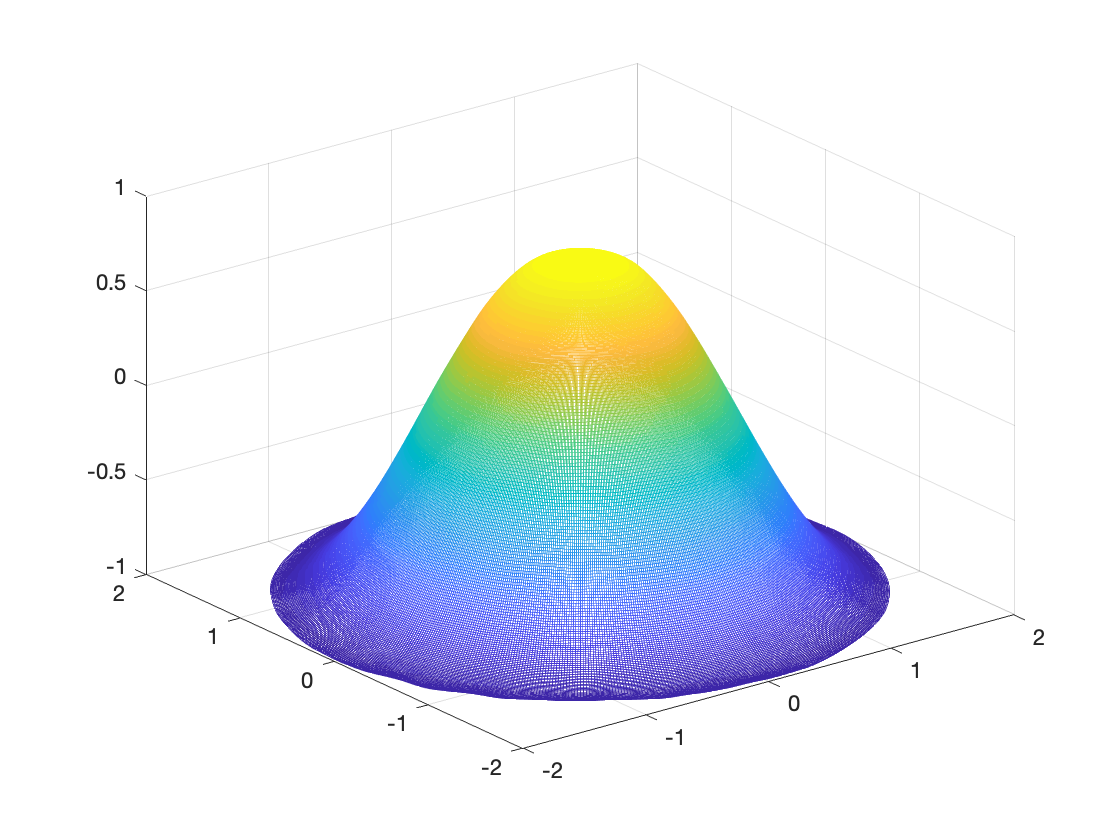}
\caption{Numerical solution of the 2D Poisson-Boltzmann equation obtained by RGA with $n=4096$ for {\bf Example 6}.}\label{fig:ex12}
\end{figure}

\end{example}

\section{Conclusions}\label{sec:con}
The process of training neural networks is the main bottleneck in applying neural networks to solve PDEs, both in terms of the effort required to tune hyperparameters and in the computational complexity required for the training process. In order to solve the resulting highly non-convex optimization problems, typically SGD or  ADAM is used to train the neural networks. These algorithms are often difficult to properly tune and often require multiple tries and additional tricks to obtain good performance. As a result, they are computationally expensive, slow, and so far not theoretically justified, despite their impressive empirical behavior.
In this paper, we develop an efficient greedy training algorithm for neural network discretization of PDEs. Guided by a greedy setup, this innovative training algorithm dynamically builds the neural network starting from a simplified version and ending with the original network via adding basis (nodes) adaptively.
    Therefore, the corresponding sub-optimization problem is easy to solve at each iteration. By gradually increasing the complexity of the model, this new training algorithm allows us to test the convergence order numerically which is not achieved by traditional training algorithms due to the complex solution landscaping. Moreover, the new training algorithm also allows us
    to find the mesh adaptivity which is one of the advantages of the neural network discretization.

\printbibliography

\end{document}